\documentclass[11pt]{amsart}

\usepackage{amsbsy, amsmath,amsfonts,amssymb,amsthm,amscd,amssymb,latexsym,}
\usepackage[latin1]{inputenc}
\usepackage{graphicx}
\usepackage[dvips]{epsfig}
\usepackage{graphics}
\usepackage{psfrag}										 
\usepackage{cancel}

\usepackage{enumerate} 
\usepackage{indentfirst}
\usepackage{euscript}

 \theoremstyle{plain}
\newtheorem{theorem}{Theorem}
 \newtheorem{lemma}{Lemma}

\newtheorem{remark}{Remark}

\newtheorem{proposition}{Proposition}

\newcommand{\R}{\mathbb{R}}

\def\gf#1#2{g_{#1#2}}

\def\fp#1{${\mathcal F}_{#1}$}

\def\Fp#1{${\mathcal F}_{#1}$}
\def\Lp#1{${\mathcal L}_{#1} $}
\def\fuu#1#2#3{#1_{#2#3}}

\usepackage[usenames,dvipsnames]{xcolor}

  \usepackage{color}

 \pagecolor{white}

 \title[ Lines of Principal Curvature on Ellipsoids of  $\mathbb R^4$  ]{ Umbilic Singularities and  Lines of Curvature on  Ellipsoids of  $\mathbb R^4$}

 \author{   D. Lopes,  J. Sotomayor and R. Garcia }

\begin{document}

%
%

 

\begin{abstract}
The topological structure of the lines of principal curvature,  the umbilic and partially umbilic singularities  of
all tridimensional  ellipsoids    of  ${\mathbb R}^4$ is described.

\end{abstract}

\maketitle

\section{Introduction}\label{sec:Intro}

In 1796   G. Monge \cite{mon} determined the first example of
a {\em principal curvature configuration}  on a surface in  $\mathbb R^3$,  consisting
of
the umbilic points (at which the principal curvatures coincide)
and, outside them, the  foliations by the  minimal and maximal
principal curvature lines.
 See also \cite{gas, gutso}.
 This configuration was achieved for the case of the ellipsoid   with
$3$  different axes defined by $q(x,y,z)=x^2/a^2+y^2/b^2+z^2/c^2=1, \,  a>b>c>0$,
where the positive orientation is defined by the normal toward the interior, given by    $ - \nabla q$.
See the illustration in figure \ref{fig:elipr3}.
The cases of  the ellipsoids with  two  different axes (with rotational symmetry, with two umbilic points at the poles)
defined by  $q(x,y,z)=x^2/a^2+y^2/a^2+z^2/b^2= 1$
and only one axis (sphere, which is totally  umbilic)
are
illustrated in  figure \ref{fig:elipR3}.

\begin{figure}[h]
\begin{center}
    \includegraphics[scale=0.6]{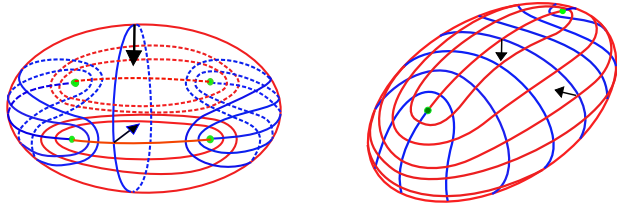}
    \caption{Two views of the global behavior of 4 umbilic
points and principal foliations $\mathcal{F}_i $ on the ellipsoid with 3 different
axes. Positive orientation given by the inner normal:
minimal (maximal) curvature lines in red (blue). When the
positive orientation is given by the outer normal, names and
colors must be interchanged. Umbilic points with $1$  umbilic separatrix, said of type $D_1$, in green.
    }
  \label{fig:elipr3}
    \end{center}
\end{figure}

 \begin{figure}[ht]
\begin{center}
    \def\svgwidth{0.6\textwidth}
    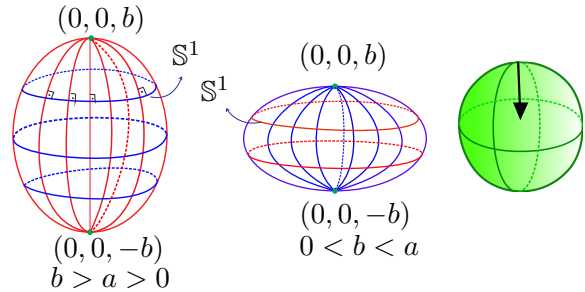
    \caption{Global behavior of the principal foliations on an  ellipsoid of revolution in $\mathbb R^3$ (left and middle).
 and the totally  umbilic sphere (right, green).
 Left:  $b>a>0$ ({\em prolate  ellipsoid}),  the leaves of the maximal principal foliation are the parallels (blue)
 and the minimal leaves are the meridians (red).
 Center: $a>b>0$ ({\em oblate ellipsoid}), the leaves of the maximal foliation are the meridians (blue) and the minimal leaves are the parallels (red). When the positive orientation is given by the outer normal  names and colors must be interchanged. Umbilic points, of center type, in green.
 Right:  $a= b = c$,  {\em sphere}, totally umbilic, green.}
  \label{fig:elipR3}
    \end{center}
\end{figure}

In this paper will  be determined  the principal configurations for
the  ellipsoids in  $\mathbb R^4$, extending in one dimension the results for $\mathbb R^3$
 outlined above and illustrated in figures  \ref{fig:elipr3} and \ref{fig:elipR3}.

By an ellipsoid in  $\mathbb R^4$ is meant the  unit  level hypersurface   $\mathbb E_{a,b,c,d}$ defined (implicitly) by a positive definite
quadratic form $Q$, which, after orthonormal diagonalization, can be written as $Q(x,y,z,w)=x^2/a^2+y^2/b^2+z^2/c^2 +w^2/d^2= 1$,
with  positive   semi-axes:$\; a,\,b,\, c,\, d. $
The positive orientation will be
defined
by the  unit inner normal  $N = \frac{-\nabla Q}{|\nabla Q|}$.

The {\em principal configuration} on an oriented hypersurface  in  $\mathbb R^4$, with euclidean scalar product $\langle \cdot, \cdot \rangle$,   consists on the
 umbilic points, at which  the $3$  principal curvatures coincide, 
 the partially umbilic points, at which only $2$  principal curvatures are equal,   and the integral foliations of the
 three principal line fields on the complement of these sets of points, which will be referred to as the { \em regular part}.

   Recall that the principal curvatures $0 < k_1 \leq k_2 \leq k_3$
 are the eigenvalues of
 the  automorphism $D(-N)$ of the tangent  bundle of  the hypersurface,  taken here as  $\mathbb E_{a,b,c,d}$  with positive unit normal $N$.

Thus the set partially umbilic points is the union of $\mathcal P_{12}$,   where $k_1 = k_2  <  k_3$,  and of   $ \mathcal P_{23}$,  where $k_3 = k_2  >  k_1$.
The set $\mathcal U$ of umbilic points is defined by  $k_1 = k_2  =  k_3$.

The eigenspaces corresponding to the  eigenvalues $k_i$  will be denoted by ${\mathcal L}_i, i= 1, 2, 3.$
 They are line fields,
 well defined
 and, for ellipsoids, also   analytic   on the {\em  regular part}.
 In fact   ${\mathcal L}_1$ is defined and analytic on the complement of
  $ \mathcal U \cup \mathcal P_{12}$,  ${\mathcal L}_3$ is defined and analytic on the complement of
$ \mathcal U \cup  \mathcal P_{23}$  and ${\mathcal L}_2$ is defined and analytic on the complement of
 $ \mathcal U \cup  \mathcal P_{12} \cup  \mathcal P_{23}$.

Given a principal direction $e_i\in {\mathcal L}_i $, consider the plane
tangent to the hypersurface passing through $q$,
having   $e_i(q)$ as the normal vector:

\begin{equation}\label{campoplanos}
\Pi_i(q)=\{(du_1,du_2,du_3);\langle(du_1,du_2,du_3),G\cdot (e_i(q))^T\rangle=0\},
\end{equation}
where $G=[g_{ij}]_{3\times3}$ is the first fundamental form 
and $(u_1,u_2,u_3)$ is a local chart.

Therefore there are three 2-plane fields which are  singular  at the umbilic and  partially umbilic sets. In general these plane fields are not
Frobenius integrable,
\cite{spivak} [vol. 1, chapter 6].
See \cite{de-so-ga}.
Special cases where these plane distributions are integrable  are
the hypersurfaces that  belong to a quadruply orthogonal system,
\cite{dar2}, \cite{eis} as is the case of all ellipsoids studied here.

The work of Garcia  \cite{garcia-tese}  established the generic properties of principal configurations on smooth hypersurfaces.   There was determined the principal configuration on  the ellipsoid  $\mathbb E_{a,b,c,d}$  with four different axes, $a > b > c> d > 0, $ in $\mathbb R^4$. It was proved that it has four closed  regular curves of  generic partially umbilic points  whose transversal structures  are of type $D_1$, as at  the umbilic points in the ellipsoid
 with 
 $3$ different axes in ${\mathbb R}^3$. See Fig. \ref{fig:elipr3}.

 A different proof of this result will be given in Theorem \ref{th:eabcd}  and 
 its conclusions   
 will be illustrated in more detail in  figures  \ref{fig:pabcd} and \ref{fig:conexao}.

 A  complete   description of the  principal configurations 
   on all  tridimensional ellipsoids is established in this paper.

Propositions \ref{prop:eaaab} and \ref{prop:eabcc2} and   Theorem \ref{th:eabcc3},  with their pertinent illustrations
establish the principal configurations on the other ellipsoids in $\mathbb R^4$.

The totally umbilic ellipsoid: $\mathbb E_{a,a,a,a}$ trivally consists in the whole sphere of radius $a$. An illustration would be the same as that in Figure \ref{fig:elipR3}, right.

 The principal configurations increase in complexity  as follows:

 The types $\mathbb E_{b,a,a,a}$ and $\mathbb E_{a,a,a,b}$, corresponding to $3$ equal axes, are studied in Proposition \ref{prop:eaaab}.

 The case of two pairs of equal axes  $\mathbb E_{a,a,b,b}$, found nowhere in the literature,
 is treated in Proposition \ref{prop:eaabb}. See ilustration in Figure \ref{fig:pumb2}.

 The ellipsoids     $\mathbb E_{a,b,c,c}$  and   $\mathbb E_{c,c,a,b}$ of one pair of equal axes, disjoint from the interval of distinct axes, $b <a$ are established in Propositions \ref{prop:eabcc1}
and \ref{prop:eabcc2}. See the illustration
 in figs. \ref{fig:pumb3caso2} and \ref{fig:pumb3caso3}.
The case $\mathbb E_{a,c,c,b}$, treating the case of the double axis inside the interval $(b,a)$ of distinct axes, exhibiting isolated umbilic points,  novel in the literature,
is proved in
 Theorem \ref{th:eabcc3}.  See the illustration
 in figs. \ref{fig:pumb3b} and \ref{fig:pumb3bg}.

Comments and references  concerning other crucial steps focusing on additional aspects of
principal configurations of hypersurfaces in  $\mathbb R^4$ have been given in section \ref{sec:CR}.

\section{Principal Configurations on Ellipsoids  in   $\mathbb R^4$ } \label{sec:PC}

\subsection {Color and Print Conventions for  Illustrations  in this Paper} \label{ss:CC}

The color convention for ellipsoids in $\mathbb R^3$ in  figures \ref{fig:elipr3} and \ref{fig:elipR3},
has been upgraded for $\mathbb R^4$  as follows.

\begin{itemize}
\item [] Black ($- \cdot -\cdot -\cdot$): integral curves of line field  $\mathcal L_1$,

 \item []Blue (\textcolor{blue}{${\mathbf {-----}} $}):  integral curves of line field  $\mathcal L_3$,

  \item [] Red (\textcolor{red}{ \rule{1.9cm}{.02cm}}): integral curves of line field $\mathcal L_2$,

 \item[] Green (\textcolor{green}{${\mathbf { -----}}$} ):  Partially umbilic arcs $\mathcal P_{12}$,

  \item[] Light Blue (\textcolor{Cyan}{${\mathbf {\mathbf{ ----}}}$}):  Partially umbilic arcs $\mathcal P_{23}$,

 \item[] Purple (\textcolor{purple}{ $\bullet $}):  Umbilic Points $\mathcal U$.

 \end{itemize}

The following dictionary has been adopted for illustrations of integral curves appearing in 
 Figs. \ref{fig:pumb1}, \ref{fig:pumb2}, \ref{fig:pumb3caso2}, \ref{fig:pumb3caso3}, \ref{fig:pumb3b}, \ref{fig:pumb3bg},  \ref{fig:cce}, \ref{fig:cce2},  \ref{fig:pabcd}  and \ref{fig:conexao} in this paper when printed in black and white:
  dashed, for blue; dotted-dashed-dotted, for black; full trace, for red.

\subsection{Open Book Structures and Hopf Bands in Ellipsoidal Principal Configurations.}
 
In Propositions \ref{prop:eaaab},  \ref{prop:eaabb},     \ref{prop:eabcc1} and  \ref{prop:eabcc2} and Theorem \ref{th:eabcc3} the ellipsoids
 
exhibit a special structure
of foliations
 with singularities by a family of two-dimensional ellipsoids, all
crossing along a circle or an ellipse, this structure is called a 
{\em  book structure}
with
 {\em binding}, in this case along the circle or ellipse, with ellipsoidal {\em pages}.
 See \cite{gir}, \cite{gir1} and \cite{wk}.

Let $M$ be a 3-manifold and $S$ a regular curve on $M$, i.e., a submanifold of codimension two.

An {\em open book structure}  on $M$ is a smooth fibration $p:M\setminus S\to\mathbb S^1$ satisfying the following  conditions:
 i) For all $\theta \in \mathbb S^1$,  the closure of $p^{-1}(\theta)$ contains $S$ and is a regular surface, a {\em page}.  ii) The submanifold $S$  coincides with  $\bigcap_{\theta\in \mathbb S^1} cl[p^{-1}(\theta)] $
     and is called the {\em binding}.

   The following geometric structure will appear in the description of partially umbilic separatrix surfaces
 in the principal configurations in sub\-sec\-tions  \ref{ss:eabcc}, \ref{ss:eccab} and \ref{ss:eaccb}.

A {\em Hopf band} in $ {\mathbb E}_{a,b,c,d} $  is an embedding   $\beta:\mathbb S^1\times [0,1]\to {\mathbb E}_{a,b,c,d}$ such that $S_1=\beta({\mathbb S}^1\times \{0\})$, $S_2=\beta({\mathbb S}^1\times \{1\})$  are linked curves with linking number equal to $\pm 1$, see \cite{gir1}.

\subsection{Three  equal axes: $\mathbb E_{a,a,a,b}$ and $\mathbb E_{b,a,a,a}$.}\label{ss:eaaab}

\begin{proposition}\label{prop:eaaab}
Consider the ellipsoid $\mathbb E_{a,a,a,b}$ defined by  $$\frac{x^2+y^2+z^2}{a^2}+\frac{w^2}{b^2}=1,\;\; a^2 \ne b^2 \ne 0.$$
Then the umbilic set   consists of two points $( 0,0,0,\pm b)$,   the partially umbilic set is the open
 set $\mathbb E_{a,a,a,b}\setminus \{( 0,0,0,\pm b)\}$.  One   principal foliation is regular on  $\mathbb E_{a,a,a,b}\setminus \{(0,0,0,\pm b)\}$
  and consists of the integral curves of the gradient of the $w-$ projection.
The distribution defined by the partially
umbilic $2-$planes has as integral foliation the level spheres of this projection.
Figure  \ref{fig:pumb1}, right,  illustrates this principal configuration. The same figure, left,
illustrates  the ellipsoid $\mathbb E_{b,a,a,a}$.

\begin{figure}[h]
\begin{center}
\def\svgwidth{0.60\textwidth}
    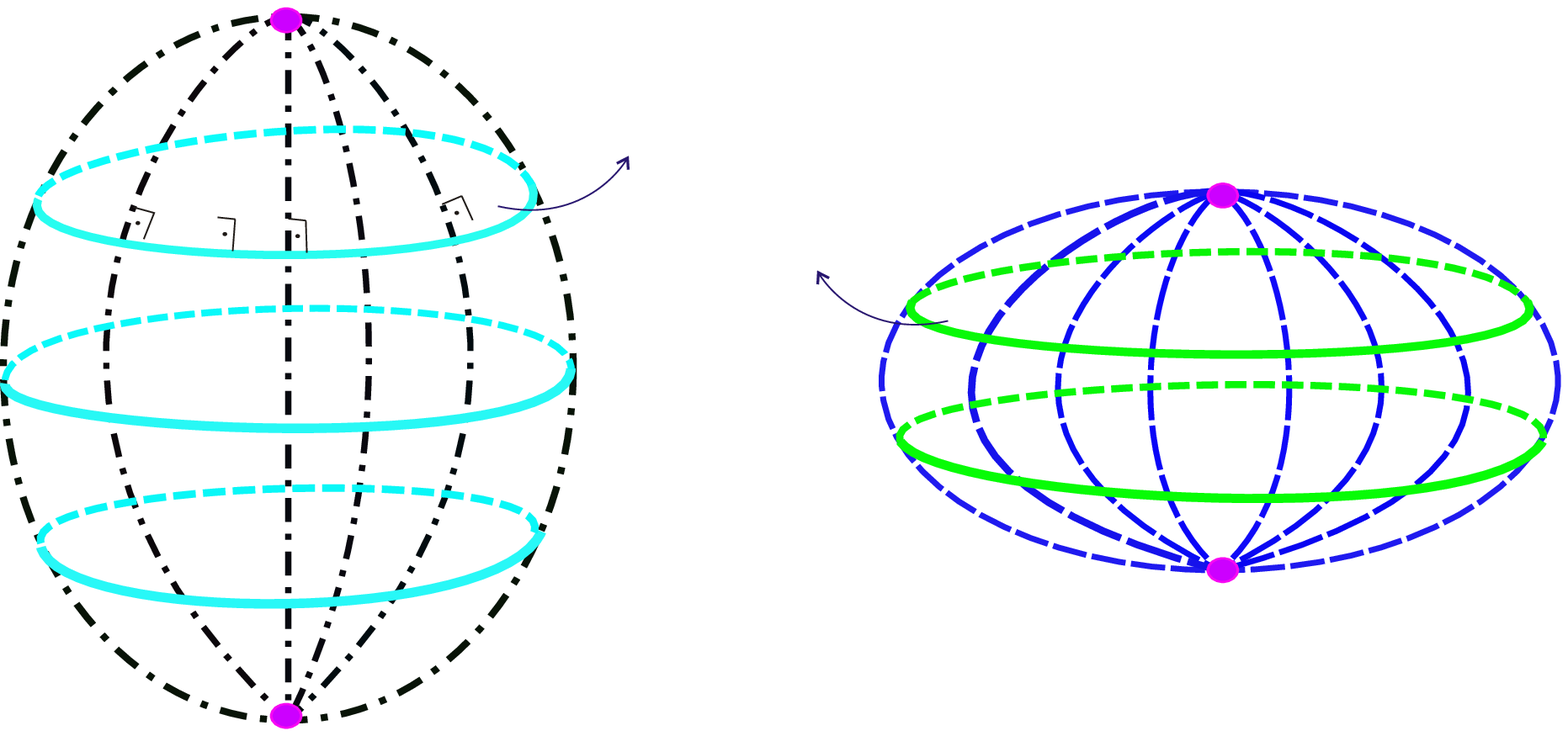
    \caption{Global behavior of the regular principal configurations  for Ellipsoids $\mathbb E_{b,a,a,a}$, left,  and $\mathbb E_{a,a,a,b}$, right. Foliations  $\mathcal{F}_1 $ (left, black $- \cdot - \cdot -$) and
      $\mathcal{F}_3$    (right, blue, ${\mathbf  ------} $).}  \label{fig:pumb1}
    \end{center}
\end{figure}
\end{proposition}

 \begin{proof} Consider the parametrization $\alpha:
[0,2\pi]\times (0,\pi)\times (-b,b)\to \mathbb E_{a,a,a,b}$ defined by:
{\small
$$\alpha(u,v,w)=(\frac{a\sqrt{b^2-w^2}}{b}\cos u\sin v,\frac{a\sqrt{b^2-w^2}}{b} \sin u \sin v, \frac{a\sqrt{b^2-w^2}}{b}\cos v, w).$$
}
The principal curvatures are given by:

$$k(u,v,w)=l(u,v,w)=\frac{b^2}{a\Delta},\;\; m(u,v,w)=\frac{ab^4}{\Delta^{3}},\; \Delta=\sqrt{(a^2-b^2)w^2+b^4}.$$
 
It follows that $k_1= k = l = k_2< k_3 = m $ when $a>b>0$ and  that $k_3 = k =k_2 = l > m = k_1$ when $0<a<b.$

The parametrization $\alpha$ does not cover the ellipse $x= 0, \, y = 0$
 in  $\mathbb E_{a,a,a,b}$ .
 
To analyze the principal configuration  around it   consider the parametrization
$$\beta(u,v,t)=(u,v,0,0)+ \sqrt{a^2-u^2-v^2}(0,0,  \cos t, \frac ba \sin t).$$
Calculation gives  the following expressions for the principal curvatures:

$$\aligned k_1(0,0,t)=&k_2(0,0,t)=\frac{b}{a\Delta}, \;\; k_3(0,0,t)=\frac{ab}{\Delta^{3}},\\
\Delta=& \sqrt{b^2\cos^2t+a^2\sin^2 t}.\endaligned$$

For $t=\pm \pi/2$ it follows that $k_1=k_2=k_3=\frac{b}{a^2}$ that correspond to the two umbilic points $(0,0,0,\pm b)$.

In the parametrization $\alpha$ the two fundamental forms $g_{ij}$ and $b_{ij}$  are diagonal and only one principal foliation is regular outside the two umbilic points, thus
it follows that the
partially
 umbilic plane field, orthogonal to the principal   regular  direction,  is integrable and the spheres
 given by $w=cte$ are
its integral  leaves.
Therefore the integral curves of the regular principal foliation are the trajectories of the gradient vector field  of the projection $\pi(x,y,z,w)=w$ with respect to the metric $g=(g_{ij})$,$(g_{ij}=\langle \partial \alpha/\partial u_i,\partial \alpha/\partial u_j\rangle )$  induced by $\alpha$.
\end{proof}

\subsection{Two pairs of equal axes: $\mathbb E_{a,a,b,b}$}\label{ss:eaabb}

\begin{proposition}\label{prop:eaabb}
Consider the ellipsoid $\mathbb E_{a,a,b,b}$ defined by  $$\frac{x^2+y^2}{a^2}+\frac{z^2+w^2}{b^2}=1,\;\; a>b>0.$$
Then it follows that:

\noindent i) The  umbilic set   is empty and the partially umbilic set is the union of two regular
circular
curves ${\mathcal P}_{23}=(a\cos u, a\sin u,0,0)$ and
${\mathcal P}_{12}=(0,0,b\cos v, b\sin v)$. The curves ${\mathcal P}_{12}$ and ${\mathcal P}_{23}$ are linked in $\mathbb E_{a,a,b,b}$.

\noindent ii) The behavior of the principal foliations near
the partially umbilic curves is   illustrated  in Fig. \ref{fig:pumb2}.
The
intermediate
foliation \Fp 2 is singular on  ${\mathcal P}_{12}\cup {\mathcal P}_{23} $, while \Fp 1 is singular
 only on  ${\mathcal P}_{12}$ and \Fp 3 is singular only on ${\mathcal P}_{23}$.

\noindent iiii)  All regular leaves of \Fp 1 and \Fp 3 are circles and the leaves of \Fp 2 are
arcs of ellipses with boundary points located at
${\mathcal P}_{12}\cup {\mathcal P}_{23}$.

\begin{figure}[h]
\begin{center}
 \def\svgwidth{0.80\textwidth}
    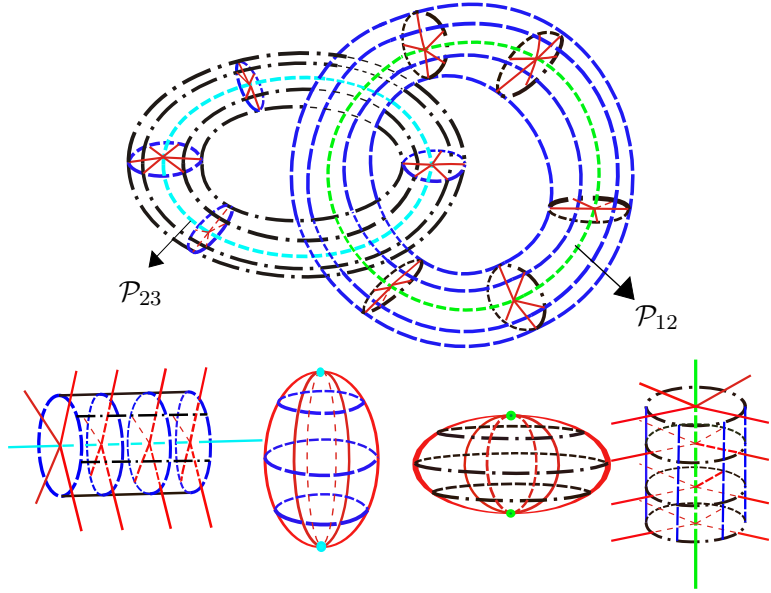
    \caption{{  
Top: Global behavior of the Principal Foliations ${\mathcal F}_i$
Bottom, center left: Ellipsoid of revolution whose poles slide along
one of partially umbilic closed line (green) and whose equator contains
the other partially umbilic closed line (light blue). They are
leaves of integral foliation of the plane distributions spanned by
$\mathcal L_2$  and $\mathcal L_3$, attaching the poles (green), illustrated locally in the
extreme left. Observe the book structure with binding along the
light blue circle, whose pages are these ellipsoids. Bottom, center
right: Ellipsoid of revolution whose poles slide along one of partially
umbilic closed line (light blue) and whose equator contains
the other partially umbilic closed line (green). They are leaves of the
integral foliation of the plane distribution  spanned by $\mathcal L_2$ and $\mathcal L_1$,
attaching the poles (light blue), illustrated locally in the extreme
right. Observe the book structure with binding along the green
circle, whose pages are these ellipsoids.
     }}
  \label{fig:pumb2}
    \end{center}
\end{figure}

\end{proposition}

\begin{proof} Consider the parametrization $$\alpha(u,v,t)=(a\cos u \cos t, a\sin u\cos t,b\cos v\sin t, b\sin v\sin t)$$
with $ 0\leq u \leq 2\pi$, $ 0\leq v \leq 2\pi$ and $ 0< t<\pi$.

A non unitary  normal   field is given by:
$$N(u,v,w)=(b\cos u\cos t, b\sin u\cos t, a\cos v\sin t, a\sin v \sin t).$$
Direct analysis shows that $\alpha$ is a regular parametrization, the coordinate curves are principal lines and  
are given by
$$\aligned k_1(u,v,t)=& \frac{b}{a\Delta } , \;\; k_2(u,v,t)=\frac{a}{b\Delta}, \\
 k_3(u,v,t)=&\frac{ab}{\Delta^3},
\;\;
\Delta= [a^2\sin^2 t+b^2\cos^2 t]^ {\frac 12}.\endaligned$$
Therefore it follows that $k_1<k_2<k_3$ when $a>b>0$ and $0<t<{\pi}$.

For $t=t_0\in (0,\pi)$ it follows that $\alpha_{t_0}(u,v)=\alpha(u,v,t_0)$ is a parametrization of a  torus invariant by two principal foliations.

For $u=u_0\in [0,2\pi] $ it follows that $\alpha_{u_0}(v,t)=\alpha(u_0,v,t)$ is a parametrization of an ellipsoid  of revolution, contained in the hyperplane $\sin u_0 x- \cos u_0 y=0$, invariant by two principal foliations. Analogous
conclusion
holds
 when $v =v_0\in [0,2\pi] $.

The ellipsoid is decomposed as follows,  the open set    $\mathbb E_{a,a,b,b}\setminus ({\mathcal P}_{12}\cup {\mathcal P}_{23})$ is foliated by tori and the two linked curves ${\mathcal P}_{12}$ and ${\mathcal P}_{23}$ are
singular leaves of this decomposition. Each torus is foliated by two one dimensional principal ones. The
intermediate
foliation \fp 2 is singular on ${\mathcal P}_{12}\cup {\mathcal P}_{23} $, while \fp 1 is singular only on ${\mathcal P}_{12}$ and \fp 3 is singular only on ${\mathcal P}_{23}$.

For $w=0$ and $w=\pi$ the parametrization above is singular.
To carry out the analysis  near the curve  ${\mathcal P}_{12}$,
given by $\alpha(u,v,0)=(a\cos u, a\sin u,0,0)$,
consider the following parametrization.
$$\bar{\alpha}(u,v,w)=(a\cos u \frac{\sqrt{b^2-v^2-w^2}}{b},a\sin u \frac{\sqrt{b^2-v^2-w^2}}{b},v,w).$$
It follows that $k_1(u,0,0)=\frac{1}a,\; k_2(u,0,0)=k_3(u,0,0)=\frac{a}{b^2}$.

The principal directions
${\mathcal L}_i(\bar{\alpha})$
are defined by the following differential equation:
$$\aligned du=&0,\;\; -vw(dv^2-dw^2)+(v^2-w^2)dvdw=0,\;\; {\mathcal L}_2(\bar{\alpha}) \;\text{and}\; {\mathcal L}_3(\bar{\alpha})\\
 dw=&0, \;dv=0,  \;\;{\mathcal L}_1(\bar{\alpha}).
  \endaligned$$
 For each $u=u_0$ fixed, $\bar{\alpha}_{u_0}(v,w)=\bar{\alpha}(u_0,v,w)$ is a surface contained in the hyperplane $\sin u_0 x-\cos u_0y=0$, invariant by two principal foliations and the principal configuration on $\alpha_{u_0}$ is equivalent to that of an ellipsoid of revolution of $\mathbb R^3$.

Similar analysis is valid to  establish the principal configuration near 
${\mathcal P}_{23}$.

 The two partially umbilic curves ${\mathcal P}_{12}$ and ${\mathcal P}_{23}$ are linked since
 they are contained in the planes $(x,y,0,0)$ and $(0,0,z,w)$, respectively.
\end{proof}

\subsection{One  pair of equal axes: $\mathbb E_{a,b,c,c}$}\label{ss:eabcc}

\begin{proposition}\label{prop:eabcc1}
Consider the ellipsoid $\mathbb E_{a,b,c,c}$ defined by  
$$
\frac{x^2}{a^2}+\frac{y^2}{b^2}+\frac{z^2+w^2}{c^2}=1,\;\; a>b>c>0.$$
Let   $E_{0}=\{(x,y,0,0): \frac{x^2}{a^2}+\frac{y^2}{b^2}  =1\}.$

Then it follows that:

\noindent i) The umbilic set   is empty and the partially umbilic set is the union of three regular curves ${\mathcal P}_{12},\; {\mathcal P}_{23}^1,\; {\mathcal P}_{23}^2$. The curves ${\mathcal P}_{23}^1$ and $ {\mathcal P}_{23}^2$  are circles contained in the planes
$(\pm a\sqrt{\frac{a^2-b^2}{a^2-c^2}},0,z,w)$.
 The pairs of curves $\{ {\mathcal P}_{12},\; {\mathcal P}_{23}^1\}$ and $\{ {\mathcal P}_{12},\; {\mathcal P}_{23}^2\}$  are linked while
 the pair
 $\{ {\mathcal P}_{23}^1,\; {\mathcal P}_{23}^2\}$  is not linked.

\noindent ii) The ellipsoid $\mathbb E_{a,b,c,c} $ is foliated  by a pencil of two dimensional surfaces, bidimensional ellipsoids, all passing through the ellipse ${\mathcal P}_{12}$.
 
Each surface  
 is an integral leaf of the distribution by planes generated by \Lp 1 and \Lp 3, and there the restricted    principal configuration is principally equivalent to that of
the bi-dimensional ellipsoid, with 3 different axes.

\noindent iii) 
The principal foliation \Fp 2 is singular on ${\mathcal P}_{12}\cup {\mathcal P}_{23}^1\cup {\mathcal P}_{23}^2$  and all its regular leaves are closed curves orthogonal to the pencil of surfaces
with book structure around ${\mathcal P}_{12}$.
The behavior of the principal foliations near
the partially umbilic curves is illustrated  in Fig. \ref{fig:pumb3caso2}.

\end{proposition}

\begin{figure}[h]
\begin{center}
    \def\svgwidth{0.8\textwidth}
    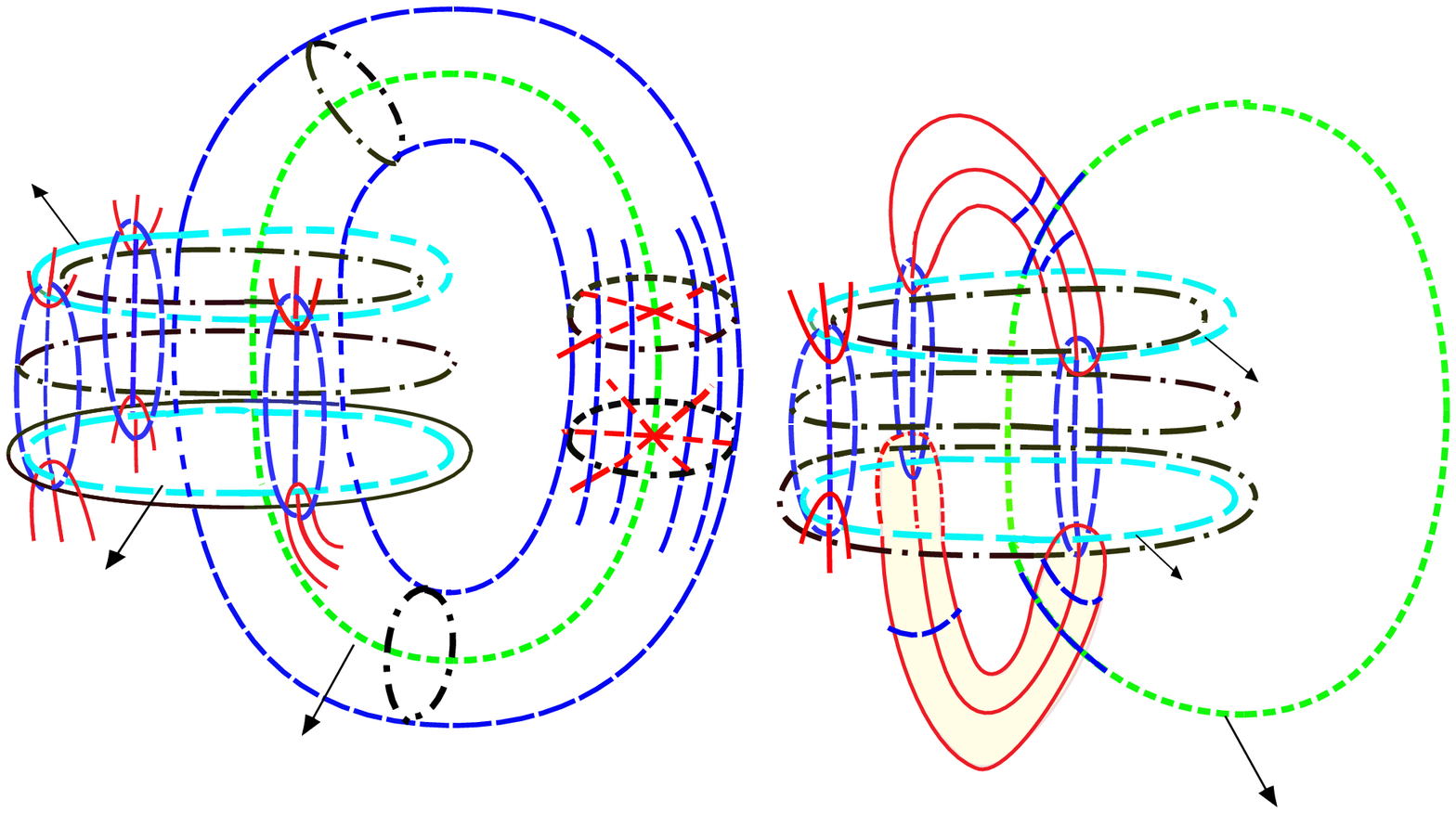
    \includegraphics[scale=0.5]{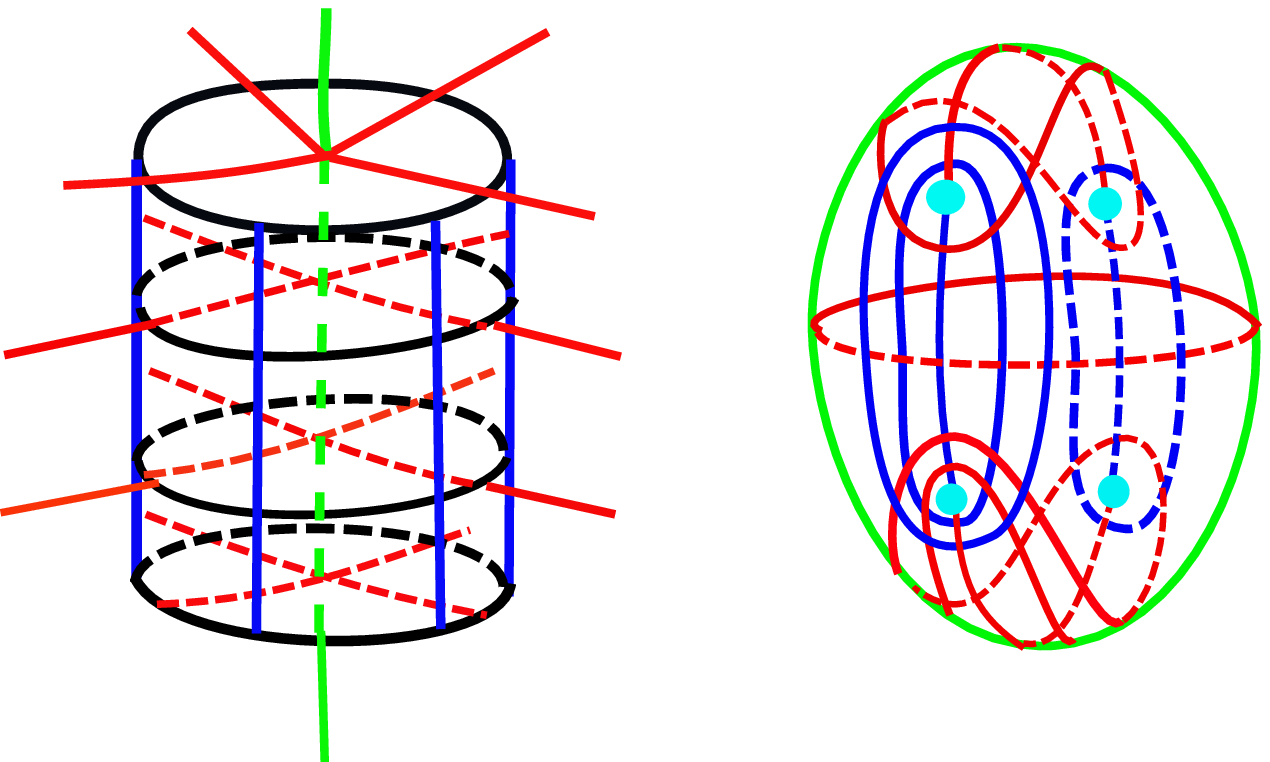}
    \caption{ 
    Behavior of the Principal Foliations  $\mathcal{F}_i  $ near the partially umbilic curves ${\mathcal P}_{23}^1, $   ${\mathcal P}_{23}^2 $ (light blue lines) and   ${\mathcal P}_{12}$ (dotted green line), top.
      Bottom, left:  Integral foliations
    by ellipsoids of plane distributions spanned by $\mathcal L_2$ and  $\mathcal L_1$ or $\mathcal L_3$,  which
    are  ellipsoids. Bottom, right: ellipsoid  with $3$ different axes  whose $4$ umbilics  slide along the  partially umbilic  closed lines  (horizontal, light  blue)  and whose equator contains  the other partially umbilic closed lines (vertical, green dotted print).
    Observe the book structure with binding along the green
circle, whose pages are these ellipsoids.
     }
  \label{fig:pumb3caso2}
    \end{center}
\end{figure}

\begin{proof} Consider the parametrization $\alpha: [0,2\pi]\times \{(v,w): v^2+w^2< 1\}\to \mathbb E_{a,b,c,c}$ defined by
$$\alpha(u,v,w)=( {a\cos u} \sqrt{1-v^2-w^2},  {b\sin u} \sqrt{1-v^2-w^2} ,cv,cw)$$
Consider a normal positive vector
$${N}(u,v,w)=-(\frac{c \cos u \sqrt{1 -v^2-w^2 }}{a}, \frac{c \sin u\sqrt{1 -v^2-w^2 }}{b}, v,w).$$
The curve $\alpha(u,0,0)$
is a partially umbilic ($k_1=k_2<k_3$). In fact,
the principal curvatures are given by:
$$\aligned k_1(u,0,0)=& k_2(u,0,0)=  \frac{ab}{\Delta^3}, \; k_3(u,0,0)= \frac{ab}{c^2\Delta},
\\
\Delta=&\sqrt{a^2\sin^2u+b^2\cos^2u}.\endaligned$$
Let 
$$\gamma_\pm(t)=(\pm a\sqrt{\frac{a^2-b^2}{a^2-c^2}},0,c\sqrt{\frac{b^2-c^2}{a^2-c^2}}\cos t, c\sqrt{\frac{b^2-c^2}{a^2-c^2}}\sin t).$$
Along $\gamma_\pm$ (green lines in Fig. \ref{fig:pumb3caso2}) the principal curvatures $k_1<k_2=k_3$ are given by:
$$k_1(t)= \frac{ab}{c^2\Delta}, \;k_2(t)=k_3(t)= \frac{a}{b\Delta}, \;\; \Delta=\frac{b}{c} \sqrt{b^2\sin^2t+c^2\cos^2t}.$$
 Next consider the parametrization $\beta:[0,2\pi ] \times  \{(u,v): \frac{u^2}{a^2}+\frac{v^2}{b^2} <1 \}\to \mathbb E_{a,b,c,c}$ defined by
 $$\beta(u,v,t)=(au,bv,c\cos t\sqrt{1- u^2-  v^2},c\sin t\sqrt{1- u^2-  v^2}).$$

 The  positive normal vector (oriented inward) ${N}_{\beta}=-\beta_u\wedge\beta_v\wedge \beta_t$ is given by:
 $$N_{\beta}=-( b c^2 u, ac^2v, abc\sqrt{1-u^2-v^2} \cos t, abc\sqrt{1-u^2-v^2} \sin t).$$
 In this parametrization the partially umbilic set is the union of two regular curves  $\beta(u_-,0,t)$, $\beta(u_+,0,t)$ where $u_{\pm}=\pm \sqrt{\frac{a^2-b^2}{a^2-c^2}}$ and $a>b>c>0$.

 The principal curvatures restricted to  these curves satisfy $k_1<k_2=k_3$ and are given by:
 $$ k_1(u_\pm,0,t)=\frac{ac}{b^3}, \; k_2(u_\pm,0,t)=k_3(u_\pm,0,t)=  \frac{a}{bc},\;\;.$$
 
The structure of the principal lines
follows from the rotational symmetry of
$\mathbb E_{a,b,c,c} $.
In fact,  consider  the pencil of hyperplanes $\pi_\theta$ given by $\cos\theta z+\sin\theta w=0$.
For all $\theta$,
$\pi_\theta \cap \mathbb E_{a,b,c,c}$ is a two dimensional ellipsoid that  contains  the ellipse $E_0$.
Moreover $\pi_\theta \cap \mathbb E_{a,b,c,c}$ is an orthogonal intersection and therefore  the principal lines of this two dimensional ellipsoid are principal lines of the three dimensional ellipsoid $\mathbb E_{a,b,c,c}$.

The third family of principal lines are the orthogonal trajectories to the
family of ellipsoids $\pi_\theta \cap \mathbb E_{a,b,c,c}$ and all are circles. It turns out that these circles are the leaves of ${\mathcal F}_2$.
\end{proof}

\subsection{One  pair of equal axes: $\mathbb E_{c,c,a,b}$}\label{ss:eccab}

\begin{proposition}\label{prop:eabcc2}
Consider the ellipsoid $\mathbb E_{c,c,a,b}$ defined by  $$\frac{x^2}{a^2}+\frac{y^2}{b^2}+\frac{z^2+w^2}{c^2}=1,\;\; c>a>b>0.$$

Let   $E_{0}=\{(x,y,0,0): \frac{x^2}{a^2}+\frac{y^2}{b^2}  =1\}.$

Then the umbilic set   is empty and the partially umbilic set is the union of three regular curves ${\mathcal P}_{23}\cup {\mathcal P}_{12}^1\cup {\mathcal P}_{12}^2$.

The partially umbilic curves  ${\mathcal P}_{12}^1$ and $ {\mathcal P}_{12}^2$  are circles contained in the hyperplanes
$ (0,\pm b\sqrt{\frac{a^2-b^2}{c^2-b^2}} ,z,w)$.
 The pairs of curves $\{ {\mathcal P}_{12}^1,\; {\mathcal P}_{23}\}$ and $\{ {\mathcal P}_{12}^2,\; {\mathcal P}_{23}\}$  are linked while
 the pair
 $\{ {\mathcal P}_{12}^1,\; {\mathcal P}_{12}^2\}$  is not linked.

Then $\mathbb E_{c,c,a,b} $ is foliated  by a pencil of two dimensional surfaces, bidimensional ellipsoids, all passing through the ellipse $E_0={\mathcal P}_{23}$.

 Each surface  
 is an integral leaf of the distribution by planes generated by \Lp 1 and \Lp 3, and there the restricted    principal configuration
is principally equivalent to that of
the bi-dimensional ellipsoid, with 3 different axes.
The principal foliation \Fp 2 is singular on $ {\mathcal P}_{12}^1\cup {\mathcal P}_{12}^2\cup {\mathcal P}_{23}$  and all its regular leaves are closed curves orthogonal to the pencil of surfaces
with book structure around ${\mathcal P}_{23}$.

The behavior of the principal foliations near
the partially umbilic curves are illustrated  in Fig. \ref{fig:pumb3caso3}.

\begin{figure}[h]
\begin{center}
 \def\svgwidth{0.7\textwidth}
    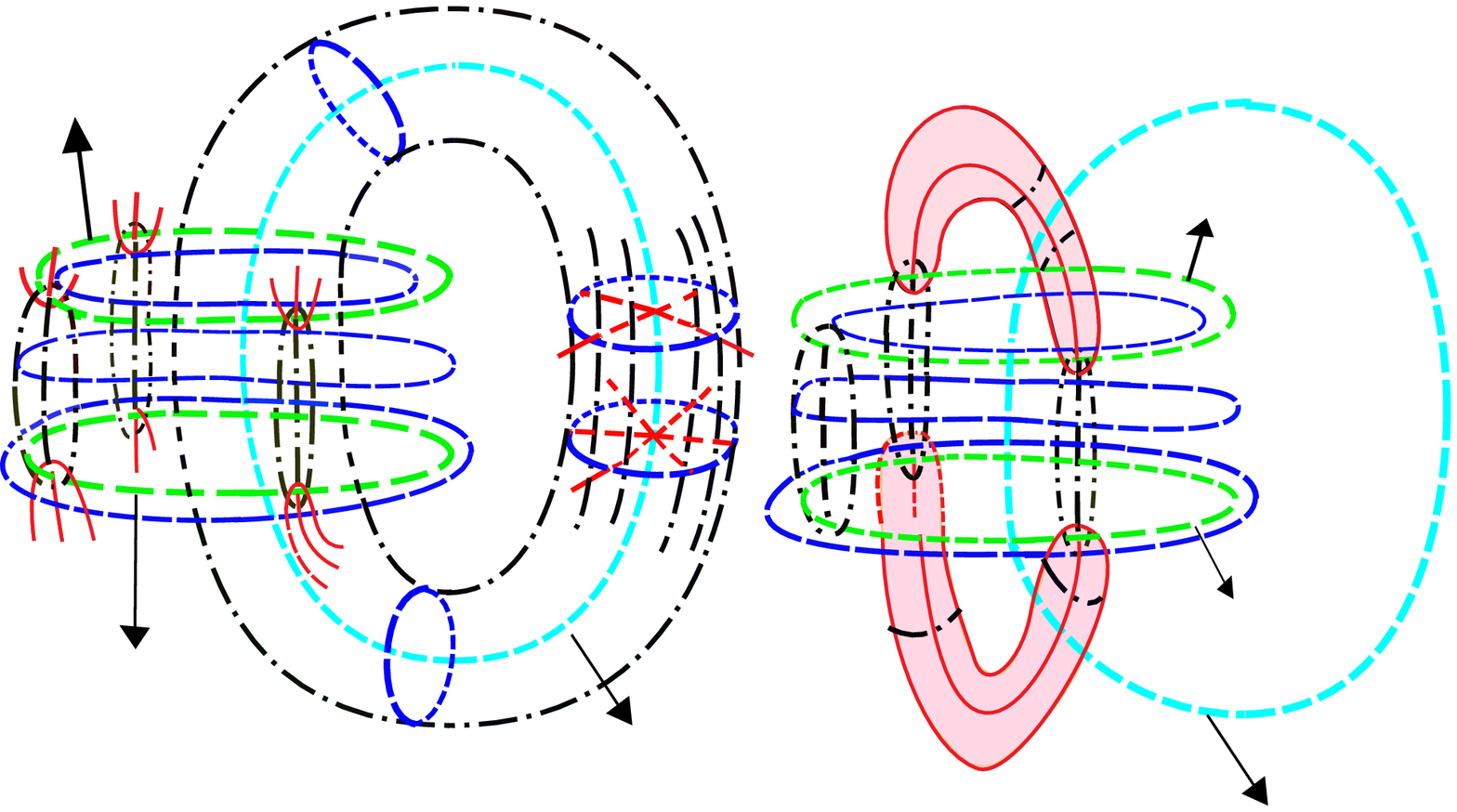
    \includegraphics[scale=0.40]{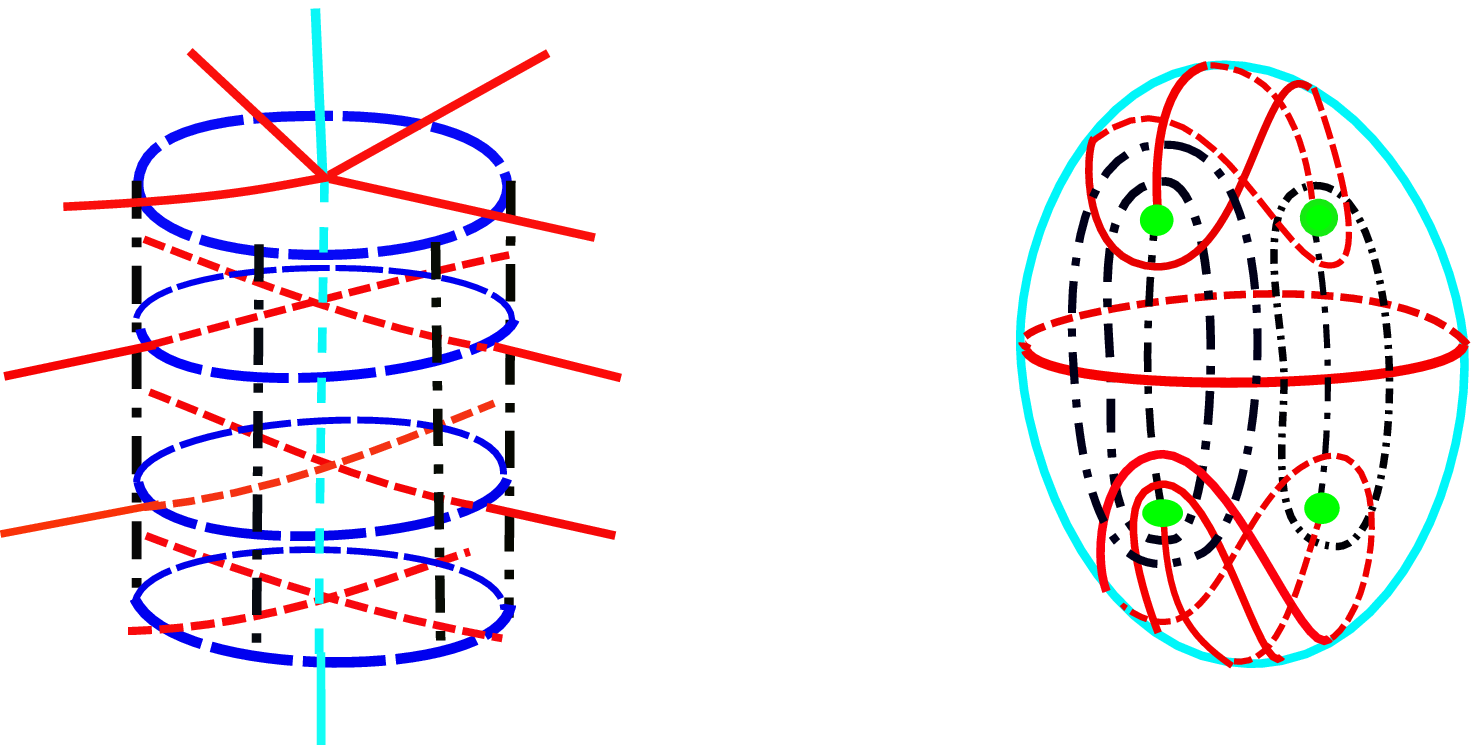}
    \caption{ Behavior of the Principal Foliations  $\mathcal{F}_i  $ near the partially umbilic curves ${\mathcal P}_{12}^1, $   ${\mathcal P}_{12}^2 $ (green lines) and   ${\mathcal P}_{23}$ (
    light blue line), top.  
    Bottom, left:  Integral foliations
    by ellipsoids of plane distributions spanned by $\mathcal L_2$ and  $\mathcal L_1$ or $\mathcal L_3$,  which
    are  ellipsoids.
    Bottom, right: ellipsoid  with $3$ different axes  whose $4$ umbilics  slide along the  partially umbilic  closed lines  (horizontal, green dotted print)  and whose equator contains  the other partially umbilic closed lines (vertical, light blue
    dotted).
    Observe the book structure with binding along the blue
circle, whose pages are these ellipsoids.
     }
  \label{fig:pumb3caso3}
    \end{center}
\end{figure}

\end{proposition}

\begin{proof} Consider the parametrization $\alpha: [0,2\pi]\times \{(v,w): v^2+w^2< 1\}\to \mathbb E_{c,c,a,b}$ defined by

$$\alpha(u,v,w)=( a\cos u \sqrt{1 -v^2-w^2},  b\sin u \sqrt{1 -v^2-w^2} , cv, cw)$$

A positive non unitary normal vector is given by:

$${N}(u,v,w)=-(\frac{c \cos u \sqrt{1 -v^2-w^2 }}{a}, \frac{c \sin u\sqrt{1 -v^2-w^2 }}{b}, v,w).$$

In this parametrization the curve
$\alpha(u,0,0)=(a\cos u, b\sin u,0,0)$
is a partially umbilic line ($k_1<k_2=k_3 $) and  the principal curvatures are given by:

$$\aligned k_1(u,0,0) =& \frac{ab}{c^2 \Delta}, \;\;k_2(u,0,0)=
k_3(u,0,0)= \frac{ab}{\Delta^3}\\
\Delta=&\sqrt{a^2\sin^2u+b^2\cos^2u}.\endaligned$$

The set defined by $cos(u)=0$ and $v^2+w^2=\frac{c^2(c^2-a^2)}{c^2-b^2}$ is a partially umbilic line.

Let $$\gamma_\pm(\theta)=(0, \pm b\sqrt{\frac{a^2-b^2}{c^2-b^2}}, c\sqrt{\frac{c^2-a^2}{c^2-b^2}}\cos \theta, c\sqrt{\frac{c^2-a^2}{c^2-b^2}}\sin \theta).$$

Along $\gamma_\pm$ the principal curvatures $k_1=k_2<k_3$ are given by:

$$k_1(\theta)= \frac{ab}{c^2\Delta}, \, k_2(\theta)=k_3(\theta)= \frac{a}{b\Delta}, \;\; \Delta=\frac{b}{c} \sqrt{b^2\sin^2\theta+c^2\cos^2\theta}.$$


Consider the parametrization $\beta:[0,2\pi ] \times  \{(u,v): u^2+v^2 <1 \}\to \mathbb E_{c,c,a,b}$ defined by

 $$\beta(u,v,t)=(au,bv,c\cos t \sqrt{1- u^2-  v^2},c\sin t \sqrt{1- u^2-  v^2}).$$

 In this parametrization the partially umbilic set is the union of two regular curves  $\beta(0,v_-,t)$, $\beta(0,v_+,t)$ where $v_{\pm}=\pm \sqrt{\frac{a^2-b^2}{c^2-b^2}}$.

Since the umbilic set is empty, the structure of the principal lines can be explained as in the proof of Proposition \ref{prop:eabcc1}.
\end{proof}

%
%

\subsection{One  pair of equal axes: $\mathbb E_{a,c,c,b}$}\label{ss:eaccb}

\begin{theorem}\label{th:eabcc3}
Consider the ellipsoid $\mathbb E_{a,c,c,b}$ defined by
  $$
\frac{x^2}{a^2}+\frac{y^2}{b^2}+\frac{z^2+w^2}{c^2}=1,\;\; a>c>b>0.$$
 Let   $E_{0}=\{(x,y,0,0): \frac{x^2}{a^2}+\frac{y^2}{b^2}  =1\}.$
Then it follows that:

\noindent i) The umbilic set is contained in the
ellipse $E_0$  and  consists of four points $(a\cos u_\pm, b\sin u_\pm ,0,0)$, $\cos u\pm= \sqrt{\frac{a^2-c^2}{a^2-b^2}}$;  the partially umbilic set consists   of the  four complementary  arcs contained in $E_0$.

\noindent ii)  The ellipsoid   $\mathbb E_{a,c,c,b} $ is foliated  by a pencil of two dimensional surfaces, bidimensional ellipsoids, all passing through the ellipse $E_0$.
Each surface 
 is an integral leaf of the distribution by planes generated by \Lp 1 and \Lp 3, and there the restricted    principal configuration
is principally equivalent to that of
the bi-dimensional ellipsoid, with 3 different axes.

\noindent iii)  
The principal foliation \Fp 2 is singular on $E_0$ and all its regular leaves are circles orthogonal to the pencil of surfaces
with book structure around $E_0$.
The behavior of the principal foliations near
the partially umbilic curves and umbilic points are illustrated  in Fig. \ref{fig:pumb3b}
and Fig. \ref{fig:pumb3bg}.
\end{theorem}

\begin{figure}[h]
\begin{center}
\def\svgwidth{0.70\textwidth}
    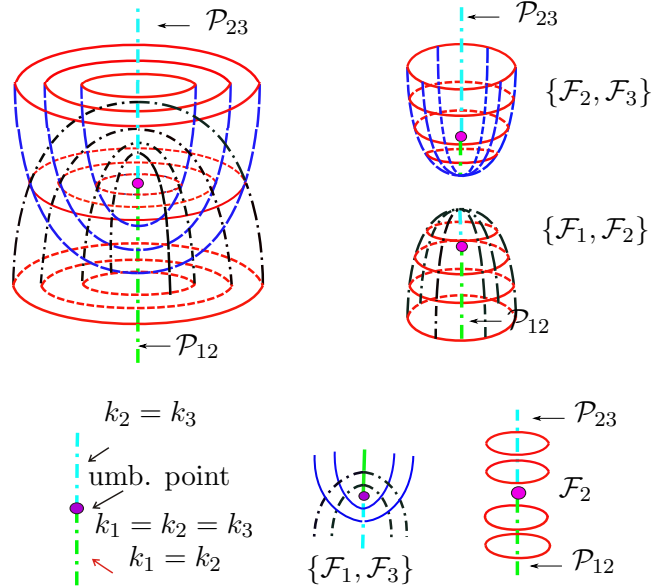
    \caption{Behavior of the  three principal foliations  $\mathcal{F}_i  $ near the partially umbilic 
    and umbilic points located on the ellipse $E_0$. 
   }
 \label{fig:pumb3b}
    \end{center}
\end{figure}

\begin{figure}[h]
\begin{center}
\def\svgwidth{0.6\textwidth}
    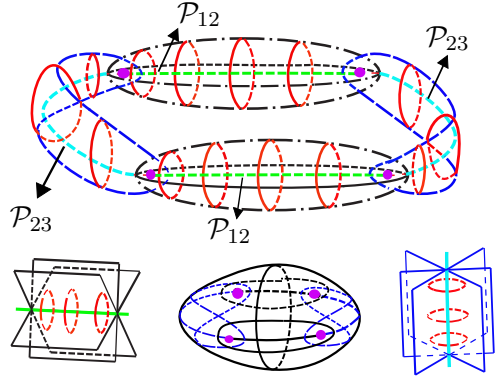
    \caption{Global behavior of the   principal foliations  $\mathcal{F}_i.$ The foliation ${\mathcal F}_2 $ is singular along the ellipse $E_0$ and all its leaves are circles. 
    Observe the book structure with binding along the ellipse $E_0$,  the  pages are two dimensional ellipsoids with three distinct axes. The ellipse   contains four umbilic points and four arcs of partially umbilic points.
    The plane   distributions spanned by ${\mathcal L}_2$ and ${\mathcal L}_1$ or ${\mathcal L}_3$   are integrable. }
  \label{fig:pumb3bg}
    \end{center}
\end{figure}

 \begin{proof} 
 Consider the parametrization $\alpha: [0,2\pi]\times \{(v,w): v^2+w^2< 1\}\to \mathbb E_{a,c,c,b}$ defined by
$$\alpha(u,v,w)=( a\cos u \sqrt{1-v^2-w^2},  b\sin u \sqrt{1-v^2-w^2} ,cv, cw).$$
In this chart the first fundamental form is given by:
$$\aligned g_{11}=&   (1-v^2-w^2)(b^2\cos^2u+a^2\sin^2u),\\
 g_{12}=&  (a^2-b^2)v \sin u\cos u,\;
  g_{13}=   (a^2-b^2)w \sin u\cos u\\
 g_{22}=&c^2+\frac{v^2(a^2\cos^2u+b^2\sin^2u)}{(1-v^2-w^2)},\;\;
 g_{23}= \frac{vw(a^2\cos^2u+b^2\sin^2u)}{ (1-v^2-w^2)}\\
 g_{33}=& c^2+\frac{w^2(a^2\cos^2u+b^2\sin^2u)}{(1-v^2-w^2)}\endaligned$$
 With respect
 to the positive normal vector (oriented inward) 
$${N}(u,v,w)=-(\frac{c \cos u \sqrt{1 -v^2-w^2 }}{a}, \frac{c \sin u\sqrt{1 -v^2-w^2 }}{b}, v,w).$$
 the second fundamental form $b_{ij}=\langle\alpha_{ij},{N}\rangle$ ($\alpha_{11}=\alpha_{uu},\; $  $\alpha_{12}=\alpha_{uv},\; $ etc.) is given by
 $$\aligned b_{11}=&c (1-v^2-w^2),\;\; b_{12}=0,\;\;\; b_{13}=0\\
 b_{22}=&  \frac{c(1-w^2)}{c^2-v^2-w^2},\; \; b_{23}= c\frac{vw}{1-v^2-w^2}\\
 b_{33}=&   \frac{c(1-v^2)}{c^2-v^2-w^2}\endaligned$$
 Let $p(k)=\det(b_{ij}-kg_{ij})$. Since $a>c>b$, write $a^2=c^2+s, \; c^2=b^2+t$ with $s>0$ and $t>0$.

 Also let $v=R\cos \gamma $ and $w=R\sin \gamma$.

 The polynomial $p(k)$ has double roots on $v=0,\;w=0$ and
 also
 when $\text{resultant}(p(k),p^\prime(k),k)=0$.

 Recall that the {\em resultant}  of a cubic polynomial $P(k)=Ak^3+Bk^2+Ck+1$ and
its derivative $P^\prime(k)=3Ak^2+2Bk+C$,
called the   {\em discriminant} of $P(k)$,
is given by  (see\cite{BP}):
$$27 A^2-18 A B C-B^2 C^2+4 B^3+4 A C^3.$$
 Algebraic manipulation shows that the zeros of the discriminant of $p(k)$ is   equivalent to the equation:

 \begin{equation} \aligned PU(R,u)=&  (s\cos^2 u-t\sin^2 u)^2 R^4+2[  (s+t )^2\cos^2 u\sin^2 u+  st ]R^2\\
 +& (t\cos^2 u-s\sin^2 u)^2=0.\endaligned
 \end{equation}

 Under  the hypothesis $(a>c>b \Leftrightarrow s>0, t>0)$ the equation $PU(R,u)=0$ has   real roots only when $R=0$ and so $p(k)$ has double or triple roots only on $v=0, w=0$.

Thus the  partially umbilic set is contained in  $\{v=0,w=0\}$ and there  the principal curvatures are given by:
$$\aligned k_1(u,0,0)=& k_2(u,0,0)=\frac{ab}{c^2\Delta} <
k_3(u,0,0)=\frac{ab}{\Delta^3}, \;\text{if} \;\; \cos^2u <\frac{a^2-c^2}{a^2-b^2} \\
k_1(u,0,0)=& \frac{ab}{\Delta^3} <
 k_2(u,0,0)=k_3(u,0,0)= \frac{ab}{c^2\Delta}, \;\text{if}\;\; \cos^2u >\frac{a^2-c^2}{a^2-b^2}  \\
\Delta=&\sqrt{a^2\sin^2u+b^2\cos^2u}.\endaligned$$

It follows that $k_1=k_2=k_3(u,0,0)$ when $\cos^2u=\frac{a^2-c^2}{a^2-b^2}=\frac{s}{s+t}$.

Next consider the parametrization $\beta: [0,2\pi]\times \{(u,v): u^2 +v^2 < 1\}\to \mathbb E_{a,c,c,b}$ defined by
$$\beta(u,v,\theta)=(au,bv,c\cos \theta\sqrt{1-u^2 -v^2 },c\sin \theta\sqrt{1-u^2 -v^2 }).$$
Similar analysis gives that the umbilic and partially umbilic points are defined by the equation

\begin{equation}
\aligned
PU(u,v,\theta)=&(su^2-tv^2)^2-2(s+t)(su^2+tv^2)+(s+t)^2=0, \; u^2+v^2<1.\\
=& (1-v^2)^2 s^2+(1-u^2)^2 t^2+2 (1-u^2 v^2-v^2-u^2)st=0\endaligned
\end{equation}

It follows that
this  equation above does not have real solutions.  In fact, in the region $ u^2+v^2<1$ the above equation, seen as a    quadratic form  in the variables $(s,t)$, is positive definite. \end{proof}

\subsection{Four distinct  axes: $\mathbb E_{a,b,c,d}$}\label{ss:eabcd}

\vskip .3cm

\noindent In this section will be established  the global behavior of principal lines in the ellipsoid of four different axes.

\begin{lemma}\label{lem:cartaprincipal}
The ellipsoid $Q_0=\mathbb E_{a,b,c,d}$  given by
$$ Q(x,y,z,w)=\dfrac{x^2}{a^2 } +\dfrac{ y^2}{b^2 } +\dfrac{
z^2}{c^2 } + \dfrac{w^2}{d^2 } - 1=0, \;\;\; a>b>c>d>0,$$ has
sixteen  principal charts $(u,v,t)=\varphi_i(x,y,z,w)$, where
$$\varphi_i^{-1}:  (d^2,c^2)  \times (c^2,b^2) \times (b^2,a^2)
\{(x,y,z,w): xwyz\ne 0\}\cap Q^{-1}(0) $$
 is defined by equation
\eqref{eq:pchart1}.

\begin{equation}\label{eq:pchart1} \aligned x^2 &= \dfrac{a^2(a^2 - u)(a^2 - v)(a^2 - t )}{(a^2 - b^2)(a^2 - c^2)(a^2 -
d^2)},
\quad y^2 = \dfrac{b^2(b^2 - u)(b^2 - v)(b^2 - t )}{(b^2 - a^2)(b^2 - c^2)(b^2 - d^2)}\\
z^2 &=\dfrac{ c^2(c^2 - u)(c^2 - v)(c^2 -t )}{(c^2 - a^2)(c^2 - b^2)(c^2 - d^2)},\quad
  w^2 = \dfrac{d^2(d^2 - u)(d^2 - v)(d^2 - t )}{(d^2 - a^2)(d^2 - b^2)(d^2 - c^2)}.
\endaligned \end{equation}

For all $p\in \{(x,y,z,w): xyzw \ne
0\}\cap Q^{-1}(0)$ and $Q_0=Q^{-1}(0)$ positively oriented, the principal curvatures satisfy
$0<k_1(p)<k_2(p)<k_3(p)$.

\end{lemma}

\begin{proof} It will be shown that  the ellipsoid $Q_0=\mathbb E_{a,b,c,d}$ belongs to a quadruply orthogonal family of quadrics.
For $p = (x,y,z,w) \in \mathbb R^4$ and  $ \lambda \in  \mathbb R$,  let
$$ Q(p,\lambda) =  \dfrac{x^2}{(a^2  -\lambda)} +\dfrac{ y^2}{(b^2 -\lambda)} +\dfrac{
z^2}{(c^2  -\lambda)} + \dfrac{w^2}{(d^2  -\lambda)}.  $$
For   $p = (x,y,z,w) \in Q_0\cap \{(x,y,z,w): xyzw\ne 0\} $  let $ u, v $  and
 $ t $ be the solutions of the quartic equation in $\lambda$,
 $Q (p,\lambda ) - 1 = 0 $
with $ u \in ( d^2 , c^2)$,  $ v \in ( c^2 , b^2)$ and
$t\in ( b^2 , a^2).\;$ 
 So 
the map $\varphi: Q_0\cap \{(x,y,z,w): xyzw\ne 0\} \to  (d^2,c^2)  \times
(c^2,b^2) \times (b^2,a^2)$ 
is well defined.

  By definition of $(u,v,t)$ it follows that:

\begin{equation}\label{eq:eqo4}
\aligned 
Q(p,\lambda) - 1 =& \dfrac{-\lambda (\lambda - u )(\lambda -
 v )(\lambda - t )}{\xi (\lambda )}   \\
  \xi (\lambda )=& ( a^2  -
 \lambda ) ( b^2 - \lambda )( c^2 - \lambda )( d^2  - \lambda
 ) \endaligned \end{equation}

Differentiating  equation \eqref{eq:eqo4} with respect  to  $\lambda$   and evaluating
 at  $ \lambda = 0 $, $\lambda = u, $ $\lambda=v$ and $\lambda=t$,
  which are the four simple roots of  the  equation, 
it follows that:
$$\aligned  \dfrac{uvt}{a^2b^2c^2d^2} &= \dfrac{x^2}{ a^4} +
\dfrac{y^2}{ b^4} + \dfrac{z^2}{ c^4} +
\dfrac{w^2}{ d^4}  \quad \text{and}\\
 \dfrac{- u (u - v ) (u - t )}{\xi (u)}  &= \dfrac{x^2}{(a^2 - u)^2 } +
\dfrac{y^2}{  (b^2 - u)^2} + \dfrac{z^2}{ (c^2 - u) ^2 } +
\dfrac{w^2}{ (d^2- u)^2 }\\
 \dfrac{- v (v - u ) (v - t )}{\xi (v)}  &= \dfrac{x^2}{(a^2 - v)^2 } +
\dfrac{y^2}{  (b^2 - v)^2} + \dfrac{z^2}{ (c^2 - v) ^2 } +
\dfrac{w^2}{ (d^2- v)^2 }\\
 \dfrac{- t (t - u ) (t - v )}{\xi (t)}  &= \dfrac{x^2}{(a^2 - t)^2 } +
\dfrac{y^2}{  (b^2 - t)^2} + \dfrac{z^2}{ (c^2 - t) ^2 } +
\dfrac{w^2}{ (d^2- t)^2 }.
\endaligned $$
 
The solution of the linear system above in the variables $x^2$, $y^2$, $z^2$ and $w^2$ is given by:
\begin{equation}\label{eq:pchart} \aligned x^2 &= \dfrac{a^2(a^2 - u)(a^2 - v)(a^2 - t )}{(a^2 - b^2)(a^2 - c^2)(a^2 -
d^2)},
\quad y^2 = \dfrac{b^2(b^2 - u)(b^2 - v)(b^2 - t )}{(b^2 - a^2)(b^2 - c^2)(b^2 - d^2)}\\
z^2 &=\dfrac{ c^2(c^2 - u)(c^2 - v)(c^2 -t )}{(c^2 - a^2)(c^2 - b^2)(c^2 - d^2)},\quad
  w^2 = \dfrac{d^2(d^2 - u)(d^2 - v)(d^2 - t )}{(d^2 - a^2)(d^2 - b^2)(d^2 - c^2)}.
\endaligned \end{equation}

The map 
$$\varphi: Q_0 \cap \{(x,y,z,w): xyzw\ne 0\}\to (d^2,c^2)  \times
(c^2,b^2) \times (b^2,a^2), \; $$
$\varphi (x,y,z,w) = ( u,v,t) $ is a regular covering
which defines a chart
in each orthant of the ellipsoid  $ Q_0$.

So, equations in 
 \eqref{eq:pchart}   define  parametrizations $\psi(u,v,t)=(x,y,z,w) $ of the  connected  components of the region $Q_0 \cap \{(x,y,z,w): xyzw\ne 0\}$.
 By symmetry, it is sufficient to consider only the positive octant $\{(x,y,z,w): x> 0, y>0, z>0, w>0\}$.

Consider the parametrization $ \psi(u,v,t)=\varphi^{-1}(u,v,t)=(x,y,z,w)$,
with $(x,y,z,w)$ in the positive orthant.  The fundamental forms of
$Q_0$ will be evaluated and expressed in terms of the function

\begin{equation}\label{eq:xi} \xi (\lambda ) = ( a^2  -
\lambda ) ( b^2 - \lambda )( c^2 - \lambda )( d^2  - \lambda
).\;\end{equation}

Evaluating $\gf 11 = (x_u)^2 + (y_u)^2 + (z_u)^2 + (w_u)^2 $,
$\gf 22 = (x_v)^2 + (y_v)^2 + (z_v)^2 + (w_v)^2 $ and
$\gf 33 = (x_t)^2 + (y_t)^2 + (z_t)^2 + (w_t)^2 $  and observing that
$\gf ij = 0,~i\not= j $,  it follows that the first fundamental form is given by:
{\small
\begin{equation}\label{eq:Ipchart1}
 I = -\frac 14\left[\dfrac{u (u - v)(u - t)}{\xi (u)}du^2 +
\dfrac{v (v - u)(v - t)}{\xi (v)}dv^2 + \dfrac{t (t - u)(t - v)}{\xi (t)}
dt^2 \right].\end{equation}
}

The normal vector ${N}=(\psi_u\wedge\psi_v\wedge
\psi_t)/|\psi_u\wedge\psi_v\wedge \psi_t|$   is oriented inward.
In fact, it follows that   $\langle {N}(u,v,t), \psi(u,v,t)\rangle <0.$

Similar and straightforward calculation shows that the second fundamental form with respect to $N$ is given by:
{\small
\begin{equation}\label{eq:IIpchart1} II = -\frac 14 \dfrac{ abcd }{(uvt)^{1\over
2}}\left[\dfrac{ (u - v)(u - t)}{\xi (u)}du^2 +
\dfrac{ (v - u)(v - t)}{\xi (v)}dv^2 + \dfrac{ (t- u)(t - v)}{\xi (t)}
dt^2 \right].\end{equation}
}
Therefore the  coordinate  lines are principal curvature lines and
  the principal curvatures
 $  b_{ ii} /\gf ii, \; (i =1,2,3)$ are given by:
$$l = \frac 1u \left(\dfrac{abcd}{\sqrt{ uvt } }\right),\quad
m = \frac 1v \left(\dfrac{abcd}{\sqrt{ uvt } }\right),\quad
n = \frac 1t \left(\dfrac{abcd}{\sqrt{ uvt } }\right).$$
Since $u \leq v \leq t $ it follows that  $l = m $ if, and only
if, $ u = v = c^2 $. Also   $m= n $ if, and only if,   $ v = t =
b^2.\;$ Therefore, for  $p\in Q_0 \cap \{(x,y,z,w): xyzw\ne 0\}$
it follows that the principal curvatures satisfy
$k_1(p)<k_2(p)<k_3(p)$.
 \end{proof}

   \begin{lemma}\label{lem:cc}
   Consider the ellipsoid $\mathbb E_{a,b,c}$  given  in $\mathbb R^3$ by:
   $$\frac{x^2}{a^2}+\frac{y^2}{b^2}+\frac{z^2}{c^2} = 1,\;  a>b>c>0.$$
 Let
 {\small
 $$ s_1 = \frac 12 \int_{b^2}^{a^2} \sqrt{\frac{-u}{( a^2-u ) ( c^2-u )}}  du < \infty \;  \text{ and}\;
s_2=\frac 12 \int_{c^2}^{b^2}   \sqrt{\frac{   - v }{ ( a^2-v ) ( c^2-v ) }} dv < \infty.$$
}
   There exists a parametrization
   $\varphi: [-s_1,s_1]\times  [-s_2,s_2]\to { \mathbb E}_{a,b,c}\cap \{(x,y,z), y\geq 0)\}$ such that
 the principal lines are the coordinate
 curves and $\varphi$ is conformal in the interior of the rectangle.

 Moreover $ \varphi(s_1,s_2)=U_1,\;   \varphi(-s_1,s_2)=U_2,\; \varphi(-s_1,-s_2)=U_3,   $ and $  \varphi(s_1,-s_2)=U_4. $

By symmetry considerations the same result
 holds  for   the region
${\mathbb E}_{a,b,c}\cap \{(x,y,z), y\leq 0)\}$. See
Fig. \ref{fig:cce}.

\begin{figure}[ht]
\begin{center}
\def\svgwidth{0.60\textwidth}
   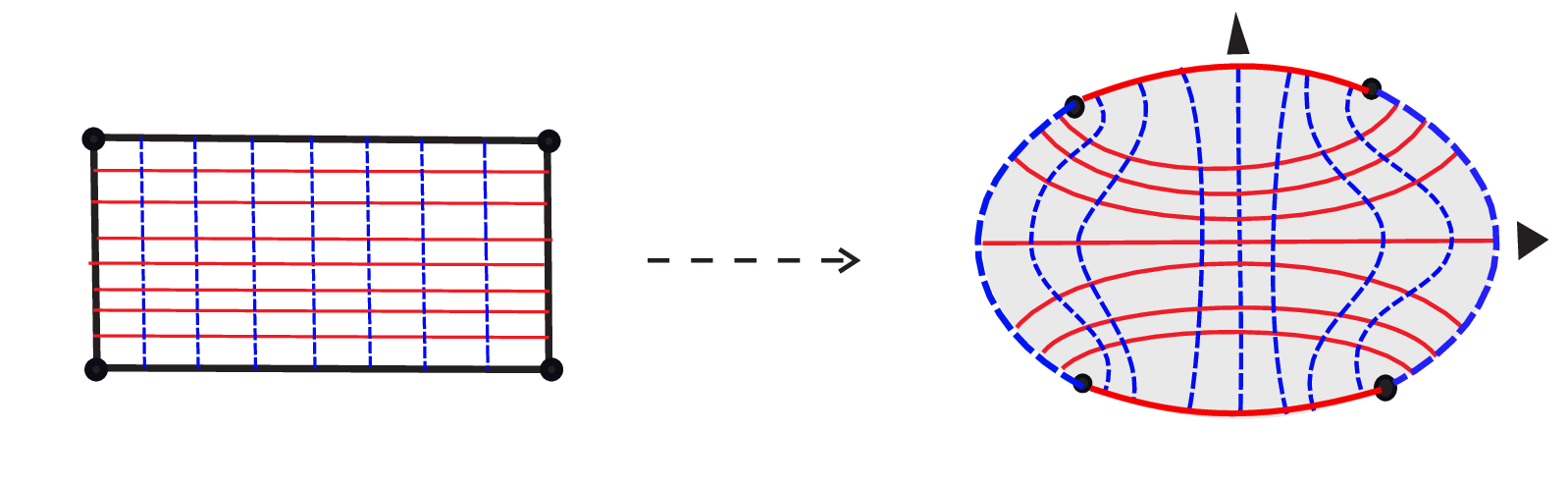
\caption{Principal lines of  the ellipsoid  $\mathbb E_{a,b,c}$ represented in a conformal chart. \label{fig:cce} }
\end{center}
\end{figure}
  \end{lemma}

  \begin{proof} The ellipsoid
 $\mathbb E_{a,b,c}$
  has a principal chart   $(u,v)$ defined by the parametrization
  $\psi:[b^2,a^2]\times [c^2,b^2]\to \{(x,y,z): x>0,\;y>0,\; z>0\}$ given by:
  $$\psi(u,v)= \left(a\sqrt{\frac{(a^2-v)(a^2-u)}{(c^2-a^2)(b^2-a^2)}}, b\sqrt{\frac{(b^2-v)(b^2-u)}{(a^2-b^2)(c^2-b^2)}},c \sqrt{\frac{(c^2-v)(c^2-u)}{(b^2-c^2) (a^2-c^2)}}\right
  ).$$
  The fundamental forms
   in this chart  are given by:
  $$\aligned I=&\; Edu^2+Gdv^2=\frac 14\frac{u(u-v)}{h(u)}du^2-\frac 14\frac{v(u-v)}{h(v)}dv^2,\\
  II=& \;edu^2+gdv^2= \frac{abc(u-v)}{4\sqrt{uv} h(u)}du^2-\frac{abc(u-v)}{4\sqrt{uv} h(v)}du^2\\
  h(t)=&( a^2-t )( b^2-t )( c^2-t ).
  \endaligned $$
  The principal curvatures are given by  $k_2(u,v)=\frac{abc}{v\sqrt{uv}},\; k_1(u,v) = \frac{abc}{u\sqrt{uv}}.$
   Therefore, $k_1(u,v)=k_2(u,v)$ if and only if $u=v=b^2$.

Considering the change of coordinates defined by
$ds_1=\sqrt{E}du$, $ds_2=\sqrt{G}dv$
 obtain a conformal
parametrization $\varphi:[0,s_1]\times [0,s_2]\to   \{(x,y,z):
x>0,\;y>0,\; z>0\}$
 in which
 the coordinate curves are principal
lines and the fundamental forms
 are given by $I=ds_1^2+ds_2^2$ and $II=k_1ds_1^2+k_2ds_2^2$.
 
 From  the symmetry of the ellipsoid
$\mathbb E_{a,b,c}$
 with respect to coordinate plane reflections,  consider an analytic continuation  of $\varphi$
 from  the rectangle $R=[-s_1,s_1]\times [-s_2,s_2]$ and
 to obtain a conformal chart $(U,V)$  of  $R$
covering  the region ${\mathbb E}_{a,b,c} \cap \{y\geq 0\}$.

 By construction $\varphi(\partial R)$ is the ellipse in
 the plane $xz$ and  the four vertices of the rectangle $[-s_1,s_1]\times [-s_2,s_2]$
 are mapped by $\varphi$  to the four umbilic points $U_i$ given by
 $\left(\pm a\sqrt{\frac{a^2-b^2}{a^2-c^2}}, 0, \pm c \sqrt{\frac{b^2-c^2}{a^2-c^2}}\right)$.
  \end{proof}

This  lemma has been included   here with a proof  for convenience of the reader. 
No explicit proof of it in the literature is known to the authors.
It will  be essential in the formulation and proof of Lemma \ref{lem:ccq2}.

   \begin{lemma}\label{lem:ccq2}
    Let $\lambda \in (d^2,c^2)$ and consider   the intersection of the quadric
    $$Q_\lambda(x,y,z,w)=\frac{x^2}{a^2-\lambda}+\frac{y^2}{b^2-\lambda}+
    \frac{z^2}{c^2-\lambda}+\frac{w^2}{d^2-\lambda}=1,\;\; a>b>c>d>0,$$
with the ellipsoid $Q_0=Q_0^{-1}(0)$.
Let $Q_\lambda= Q_\lambda^{-1}(0)\cap Q_0$.
Then $Q_\lambda$ is a  compact quartic surface with two connected components, both being  diffeomorphic to $\mathbb S^2$ and
there exists a  principal parametrization
   $\varphi_\lambda: [-s_1(\lambda),s_1(\lambda)]\times  [-s_2(\lambda),s_2(\lambda)]\to Q_\lambda \cap \{(x,y,z,w), y\geq 0, w>0)\}$ such that
 the principal lines are the coordinate curves and $\varphi$ is conformal in the interior of the rectangle.
 Here,
 $$   s_1(\lambda) = \frac 12 \int_{b^2}^{a^2} \sqrt{ \frac{(u-\lambda)u}{\xi(u)}} du < \infty \;\;\text{ and}
 \;\;\; s_2(\lambda)= \frac 12 \int_{c^2}^{b^2} \sqrt{ \frac{(v-\lambda)v}{\xi(v)}}dv < \infty,$$
 where $\xi(t)=( a^2-t ) ( c^2-t )(d^2-t)$.

 Moreover $\varphi_\lambda(s_1(\lambda),s_2(\lambda))={\mathcal P}_{12}^1(\lambda),$ \; $\varphi_\lambda(-s_1(\lambda),s_2(\lambda))={\mathcal P}_{12}^2(\lambda),$ \;  $\varphi_\lambda(-s_1(\lambda),-s_2(\lambda))={\mathcal P}_{12}^3(\lambda),$ \;
 $\varphi_\lambda(s_1(\lambda),-s_2(\lambda))={\mathcal P}_{12}^4(\lambda). $

By symmetry considerations the same result
 holds for   the region  $Q_\lambda
\cap \{(x,y,z,w), y\leq 0, w>0)\}$. See Fig. \ref{fig:cce2}. Similar
conclusions hold for  $\lambda\in (b^2,a^2)$.

\begin{figure}[ht]
\begin{center}
 \def\svgwidth{0.90\textwidth}
    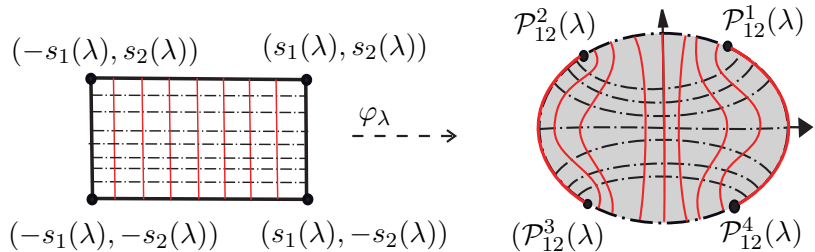
\caption{Principal lines of   the quartic $Q_\lambda$  represented in a conformal chart.  \label{fig:cce2} }
\end{center}
\end{figure}
  \end{lemma}

\begin{proof} The principal chart defined by equation \eqref{eq:pchart1} has the first and second fundamental forms given by equations \eqref{eq:Ipchart1} and \eqref{eq:IIpchart1}. Therefore, for   $\lambda \in (d^2,c^2)$,  $\psi_\lambda(u,v)=\psi(u,v,\lambda)$, $\psi:(b^2,c^2)\times (b^2,a^2)\to Q_\lambda$ is a parametrization of $Q_\lambda$
in the region $Q_\lambda
\cap \{(x,y,z,w),   w>0)\}$ by principal curvature lines. Since the quadratic form $Q_\lambda(x,y,z,w)$ has signature 1 $(+++-)$ in the    
sense of Morse, 
it follows that $Q_\lambda$
  has two connected components, one contained in the region $w>0$ and the other in the region $w<0$.
 The conclusion of the proof is similar to that of Lemma \ref{lem:cc}.
\end{proof}

   \begin{lemma}\label{lem:toro} Let $\lambda \in (c^2,b^2)$ and consider the quartic surface
 $Q_\lambda= Q_\lambda^{-1}(0)\cap Q_0$.

  Then $Q_\lambda$ is diffeomorphic to a bidimensional torus of revolution and there exists a conformal principal parametrization of $Q_\lambda$  such that
 the principal lines are the coordinates curves.

\end{lemma}

\begin{proof} Similar to the proof of Lemma \ref{lem:ccq2}. For $\lambda \in (c^2,b^2)$  the quadratic form $Q_\lambda(x,y,z,w)$ has signature 2 $(++--)$, so it follows that $Q_\lambda$ has only one connected component diffeomorphic to a torus.
\end{proof}
 
\begin{remark}
Lemmas \ref{lem:cc}, \ref{lem:ccq2} and \ref{lem:toro} establish that all regular leaves of all principal foliations ${\mathcal F}_i$ ($i=1,2,3$) of the ellipsoid ${\mathbb E}_{a,b,c,d}$ are closed (no recurrences occur) and also that the singularities (partially umbilic and umbilic points) are contained in the coordinate planes.

This geometric structure is  crucial to obtain the global description of the principal foliations of the ellipsoid ${\mathbb E}_{a,b,c,d}$ with four different axes.
Examples of recurrent principal lines on surfaces diffeomorphic to a torus or to  the ellipsoid can be made explicit, see \cite{gas}  and \cite{gsln}. 
\end{remark}

 \begin{proposition}\label{prop:dupin}
  For   $\lambda \in (b^2,a^2)\cup (d^2,c^2)$, the  principal configuration $\{ {\mathcal F}_1, {\mathcal F}_2, {\mathcal P}_{12}\}$ or $\{ {\mathcal F}_2,{\mathcal F}_3, {\mathcal P}_{23}\}$    on each connected component of $Q_\lambda=Q_{\lambda}^{1}\cup Q_{\lambda}^{2}$ is principally equivalent to an  ellipsoid with three distinct axes in $\mathbb R^3$, i.e., there exists a homeomorphism $h_i:Q_{\lambda}^i\to \mathbb E_{a,b,c}$ preserving both principal foliations and singularities.

For $\lambda\in (c^2,b^2)$, $Q_\lambda$  is  principally equivalent to a  torus of revolution in $\mathbb R^3$.
 
\end{proposition}

\begin{proof}
It follows from   Lemmas \ref{lem:cc} and \ref{lem:ccq2} that the  principal configuration of the ellipsoid with three different axes in $\mathbb R^3$ is topologically equivalent to the principal configuration of  the quartic surface  $Q_\lambda\subset E_{a,b,c,d}$ when  $\lambda\in (d^2,c^2)\cup  (b^2,a^2)$.

By Lemma \ref{lem:toro} the  principal configuration of $Q_\lambda\subset{ \mathbb E}_{a,b,c,d}$, with $\lambda\in (c^2,b^2)$,  is topologically equivalent to the principal configuration of a torus of revolution in $\mathbb R^3$.

 The construction of the homeomorphism can be done by the method of canonical regions. 
Traditionally, as explained by M. C. Peixoto and M. M. Peixoto \cite{pp},  this procedure  partitions the domain of a foliation with singularities into cells with a parallel structure and other disk-cells and annuli containing, respectively,  the singularities and periodic leaves of the foliation. These are the {\em canonical regions}. Then the homeomorphism preserving leaves between two foliations with similar partitions of canonical regions is carried out  first on the parallel regions and then extending it to the  disk-cells and annuli. 
The method of canonical regions was adapted from  to the case  of single foliations with singularities  to that of nets of transversal foliations  with singularities on surfaces, such as the principal configurations on surfaces,  in \cite{gutso}.  

In the present case the principal configuration of the ellipsoid,  see Fig. \ref{fig:cce} in lemma \ref{lem:cc}, is equivalent to that of  the quartic surface diffeomorphic to the ellipsoid, see Fig. \ref{fig:cce2} in Lemma \ref{lem:ccq2}. The canonical regions in each hemisphere are product of parallel foliations. 
 In the case of the torus, Lemma \ref{lem:toro}, both principal configurations consist  of foliations by closed curves, circles in the torus of revolution and quartic  algebraic curves on the surface $Q_\lambda$.  In this case   the canonical region is the entire surface.
\end{proof}

Take  the parametrizations 
$$\alpha_\pm (u,v,w)=(au,bv,cw,\pm d\sqrt{1-u^2-v^2-w^2})$$
of the ellipsoid $\mathbb E_{a,b,c,d}$ defined by 
 $$\frac{x^2}{a^2}
+\frac{y^2}{b^2}+\frac{z^2}{c^2}+\frac{\omega^2}{d^2}=1,\;\; a>b>c>d>0.$$

\begin{theorem} \label{th:eabcd}
The umbilic set  of  $\mathbb E_{a,b,c,d}$ is empty and its partially umbilic set  consists of four closed curves
${\mathcal P}_{12}^1, \; {\mathcal P}_{12}^2,\; {\mathcal P}_{23}^1$ and $ {\mathcal P}_{23}^3$,
that, assuming the notation
 $\;c^2=d^2+t,\; b^2=d^2+t+s,\; a^2=d^2+t+s+r,$
 
\noindent are   defined in the chart $(u,v,w)$ by:
\begin{equation} \label{eq:elip_hyp_a}
\aligned w=&0,\; s(r+s+t)u^2 + (t+s)(r+s)v^2-s(r+s)=0, \;\; \text{ellipse}\\
v=&0,\; (st+sr+s^2)u^2-tr w^2- sr=0,\;\; \text{hyperbole}.\endaligned
\end{equation}

  \noindent i)  \; The principal foliation  ${\mathcal F}_1  $
is singular on ${\mathcal P}_{12}^1\cup  {\mathcal P}_{12}^2,\;$ ${\mathcal F}_3 $
is singular on ${\mathcal P}_{23}^1\cup  {\mathcal P}_{23}^2$  and  ${\mathcal F}_2  $
is singular on ${\mathcal P}_{12}^1\cup  {\mathcal P}_{12}^2 \cup {\mathcal P}_{23}^1\cup  {\mathcal P}_{23}^2$.
 
\noindent  The partially umbilic curves ${\mathcal P}_{12}^1\cup  {\mathcal P}_{12}^2 $
 whose transversal structures are 
of type  $D_1$, the partially  umbilic 
separatrix surfaces 
of  ${\mathcal P}_{12}^1$ and $  {\mathcal P}_{12}^2$  
  span a cylinder  $\fuu C12 $ such that
$\partial\fuu C12$ $ =  {\mathcal P}_{12}^1\cup  {\mathcal P}_{12}^2$.

Also, the partially umbilic curves ${\mathcal P}_{23}^1$ and $  {\mathcal P}_{23}^2$
are of type  $D_1$, its  partially umbilic 
separatrix surfaces 
  span a cylinder  $\fuu C23 $
such that $\partial\fuu C23 =  {\mathcal P}_{23}^1\cup  {\mathcal P}_{23}^2$.

\noindent  All the leaves of    ${\mathcal F}_1  $   outside the cylinder
 $\fuu C12$ are compact and diffeomorphic to  $\mathbb S^1.\;$
 Analogously for the principal foliation ${\mathcal F}_3$ outside the cylinder $\fuu C23$. 
 See illustration in Fig. \ref{fig:pabcd} 
 
    \noindent ii)  \;  The principal foliation ${\mathcal F}_2   $ is singular at
${\mathcal P}_{12}^1\cup  {\mathcal P}_{12}^2 \cup {\mathcal P}_{23}^1\cup  {\mathcal P}_{23}^2$  and there are Hopf bands
 $ H_{123}^1$ and $ H_{123}^2 $ such that $\partial H_{123}^1 ={\mathcal P}_{12}^1\cup {\mathcal P}_{23}^1\;$ and $\partial H_{123}^2 ={\mathcal P}_{12}^2\cup {\mathcal P}_{23}^2\;$, which are partially  umbilic 
 separatrix surfaces. 
 
\noindent All the leaves of ${\mathcal F}_2$,  outside the partially umbilic surfaces separatrices, are closed.
 See illustration in Fig. \ref{fig:conexao}.

\begin{figure}[h]
\begin{center}
    \def\svgwidth{0.60\textwidth}
    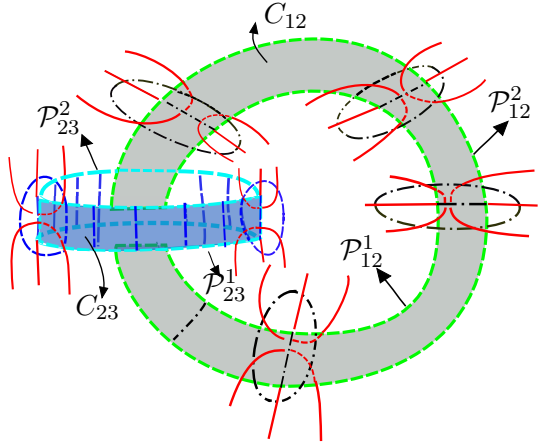
    \caption{Global behavior of the principal foliations  $\mathcal{F}_i,\; (i=1,2,3)$. The cylinder $C_{12}$ is foliated by principal lines of ${\mathcal F}_1$  with boundary two partially umbilic lines (dotted green lines). The cylinder $C_{23}$ is foliated by principal lines of $\mathcal F_3$ with boundary two partially umbilic lines (dotted blue lines). 
     }
  \label{fig:pabcd}
    \end{center}
\end{figure}

\begin{figure}[h]
\begin{center}
         \def\svgwidth{0.85\textwidth}
    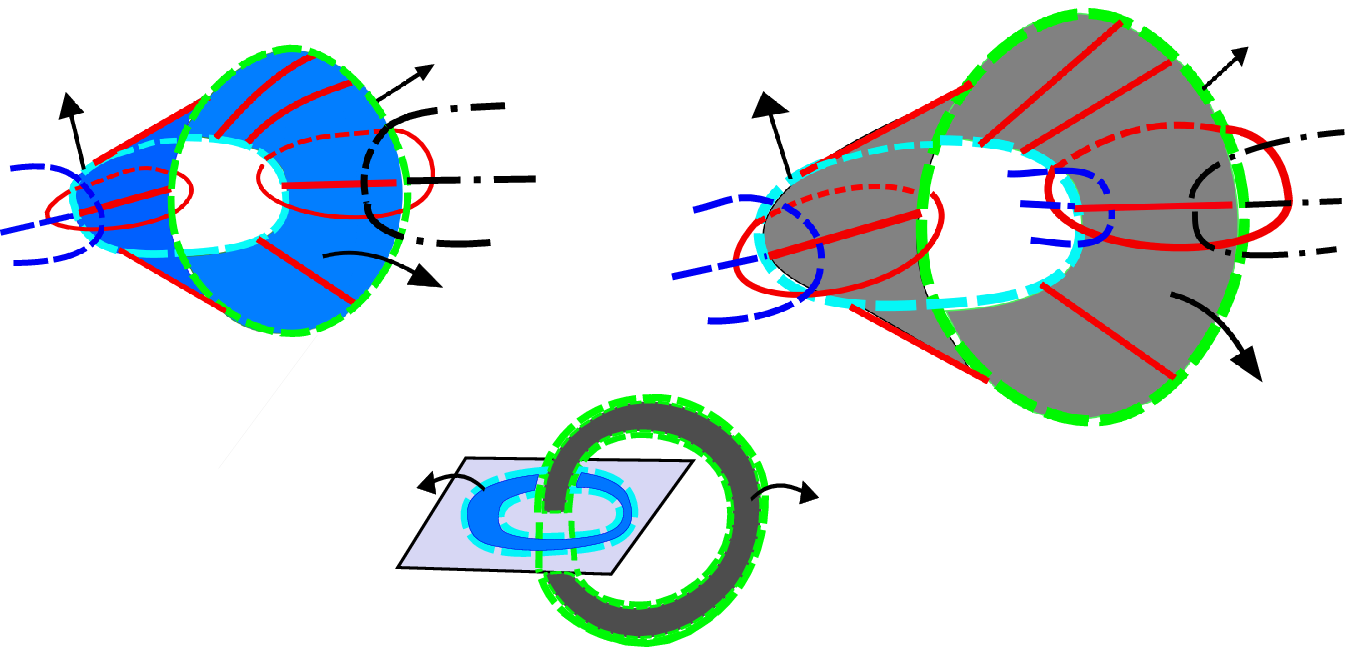
            \includegraphics[scale=0.5]{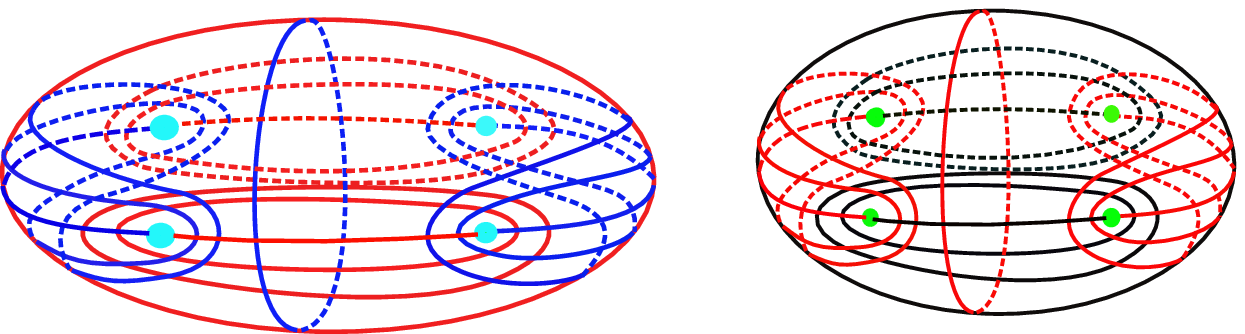}
    \caption{   Hopf bands $H_{123}^1$ and  $H_{123}^2$ with  leaves of $\mathcal{F}_2 $.
    Global behavior of the Principal Foliations  $\mathcal{F}_i  $ near 
    the partially umbilic curves, top. Bottom:  quartic surface whose $4$ umbilics  slide along the  partially umbilic  closed lines  (horizontal, blue dotted print,
    and vertical, green dotted print).
 }
  \label{fig:conexao}
    \end{center} 
\end{figure}
 
\end{theorem}
\begin{proof} It will be considered only the parametrization $\alpha_+$. The other case follows by symmetry.

In the  
parametrization
$\alpha_+,$
the first fundamental form $(g_{ij})$ is given by:
$$\aligned g_{11}=&\frac{a^2(u^2+v^2+w^2-1)-d^2u^2}{u^2+v^2+w^2-1},\;\;\;\;\; g_{12}=  -\frac{d^2uv}{u^2+v^2+w^2-1},\\
g_{22}=&\frac{b^2(u^2+v^2+w^2-1)-d^2v^2}{u^2+v^2+w^2-1},\;\;\;\;\;g_{13}=  -\frac{d^2uw}{u^2+v^2+w^2-1},\\
g_{33}=&\frac{ c^2(u^2+v^2+w^2-1)-d^2w^2}{u^2+v^2+w^2-1},\;\;\;\;\;
 g_{23}=  -\frac{d^2vw}{u^2+v^2+w^2-1}.\endaligned $$
The positive normal field (oriented inward)
$N_+=\frac{1}{abcd} \,\alpha_u\wedge \alpha_v\wedge \alpha_w$
 is given by:
$${N_+}(u,v,w)=-( \frac{ u}{a\Delta} ,  \frac{ v}{b\Delta},  \frac{ w}{c\Delta},  \frac{1}{d}), \;\;\Delta=\sqrt{1-u^2-v^2-w^2}.$$
 
The second fundamental form $(b_{ij}),$ $b_{ij}=\langle \alpha_{ij},N_+\rangle$,
relative  to $N_+$, is given by:
$$\aligned b_{11}=&
 \frac{1-v^2-w^2}{(1-u^2-v^2-w^2)^{\frac 32}},\;\;\;\;\; b_{12}=
 \frac{uv}{(1-u^2-v^2-w^2)^{\frac 32}},\\
b_{22}=&
\frac{1-u^2-w^2}{(1-u^2-v^2-w^2)^{\frac 32}},\;\;\;\;\; b_{13}=
\frac{uw}{(1-u^2-v^2-w^2)^{\frac 32}},\\
b_{33}=&
\frac{1-u^2-v^2}{(1-u^2-v^2-w^2)^{\frac 32}},
\;\;\;\;\; b_{23}=
\frac{vw}{(1-u^2-v^2-w^2)^{\frac 32}}.\endaligned$$
Calculation shows that the principal curvatures, which are the
roots of $\det (b_{ij}/|N_+|-kg_{ij})=0$,  are  given by the zeros  of

\begin{equation}\label{eq:k}
\aligned
p(k)=&  1- [(a^2-d^2) u^2+(b^2-d^2)v^2+(c^2-d^2)w^2-a^2-b^2-c^2 ]k\\
+& [ (d^2-a^2)(b^2+c^2)u^2+(a^2+c^2)(d^2-b^2) v^2\\
+& (d^2-c^2)(a^2+b^2)w^2
+  b^2 c^2+a^2 c^2+a^2 b^2]k^2\\
-&[ b^2 c^2(a^2-d^2) u^2 +a^2c^2(b^2-d^2)v^2+ a^2b^2(c^2-d^2)w^2-a^2b^2c^2]k^3.\endaligned
\end{equation}

Denote by   $R(u,v,w)=\text{resultant}(p(k),p^\prime(k),k)$, the {\em discriminant} of $p(k)$.

Consider the polynomials $p_i$ defined below.
$$\aligned
p_1=&-s[(2t+2s)v^2 -2tw^2-s]+[(s+t)v^2+t w^2]^2,\\
 p_2=&(s+t)(s+r)v^2+r tw^2  +r(r +s)\\
p_3=&[(s+t+r)u^2+tw^2]^2-(s+r)[2(t+s+r)u^2-2tw^2-s-r],\\
p_4=&s(t+s+r) u^2-r tw^2 -r s\\
p_5=&-r[2(t+s+r)u^2-2(s+t)v^2-r]+[(s+t+r)u^2+(s+t)v^2]^2,\\
 p_6=&s(s+t+r)u^2+(s+t)(s+r)v^2-s(s+r)\endaligned$$
It follows that $R(u,v,0)=-\frac 16 p_5p_6^2$, $R(u,0,w)=-\frac 16 p_3p_4^2$
and $R(0,v,w)=-\frac 16 p_1p_2^2$.

By Lemma \ref{lem:cartaprincipal} there are no umbilics   nor partially umbilic points outside the coordinate hyperplanes.

Therefore, the partially umbilic   and umbilic sets are defined by the
real
branches
of the equations $R(u,v,0)=0,\; R(0,v,w)=0,\; R(0,v,w)=0$.

Now, by  elementary analysis, it follows that the real zeros of the equations above are given by:
$$ \{(u,0,w): p_4=0\}\cup  \{(u,v,0): p_6=0\}.$$
The explicit solutions are given by:

\begin{equation*}\aligned w=&0,\; s(r+s+t)u^2 + (t+s)(r+s)v^2-s(r+s)=0, \;\; \text{(ellipse)}\\
v=&0,\; (st+sr+s^2)u^2-tr w^2- sr=0,\;\; \text{(hyperbole)}.\endaligned
\end{equation*}

Let ${\mathcal P}_{23}^1=\alpha_+(ellipse)$ and  ${\mathcal P}_{23}^1=\alpha_-(ellipse)$. It follows that ${\mathcal P}_{23}^1$ and ${\mathcal P}_{23}^2$ are   partially umbilic curves and they are  singularities of  ${\mathcal F}_2 $ and ${\mathcal F}_3$ and regular leaves of ${\mathcal F}_1$.

The hyperbole $  v= 0,\; (st+sr+s^2)u^2-tr w^2- sr=0$ is only a part of the other two
connected
 components of the partially umbilic set defined in the domain $u^2+v^2+w^2<1$.

To compute the other two
connected components of the partially umbilic set  in
a single local chart
consider the parametrizations
$$\beta_\pm (u,v,w)=(\pm a\sqrt{1-u^2-v^2-w^2},  bv , cu, dw).$$
Similar analysis as that  for $\alpha_\pm (u,v,w)$ gives that
 the partially umbilic set is defined by:

\begin{equation*}\aligned v=&0,\;  (t+s)(r+s)u^2+s(r+s+t)w^2-s(t+s)=0, \;\; \text{(ellipse)}\\
u=&0,\;rtv^2- s(t+r+s)w^2+st=0,\;\; \text{(hyperbole)}.\endaligned
\end{equation*}

Let ${\mathcal P}_{12}^1=\beta_+(ellipse)$ and  ${\mathcal P}_{12}^1=\beta_-(ellipse)$. It follows that ${\mathcal P}_{12}^1$ and ${\mathcal P}_{12}^2$ are   partially umbilic curves and they are  singularities of  ${\mathcal F}_1 $ and ${\mathcal F}_2$ and regular leaves of ${\mathcal F}_3$.

The two partially umbilic lines  ${\mathcal P}_{12}^1$ and ${\mathcal P}_{12}^2$  (respectively   ${\mathcal P}_{23}^1$ and ${\mathcal P}_{23}^2$) contained in the plane $v=0$ (respectively in the plane $w=0$)  bound  a cylinder
$C_{12}$ (respectively bound  a cylinder  $C_{23}$).  
The cylinder $C_{12}$ is diffeomorphic to the cylinder $\mathbb S^1\times [0,1]$  and the boundary curves are not linked in ${\mathbb E}_{a,b,c,d}$. It   is foliated   by principal lines of ${\mathcal F}_1$. Observe that the book structure in  Proposition 
\ref{prop:eabcc1}
   near a double partially umbilic
   green curve ${\mathcal P}_{12}$, in  Fig.  \ref{fig:pumb3caso2}
    changes 
       its transversal structure as  
    the principal configuration of an ellipsoid of revolution 
   changes into
   one with three different axes.    

 In a similar way    $C_{23}$
is diffeomorphic to $C_{12}$ and it is foliated by principal lines of 
${ \mathcal F}_3$, see Fig \ref{fig:pabcd}.

   The pair of curves  $\{ {\mathcal P}_{12}^1, {\mathcal P}_{23}^1\}$ is  linked  with linking number equal to $\pm 1$.
      This follows from Proposition \ref{prop:eaabb} and by the invariance of the linking number by homotopy, see \cite{nov}.
      
    Also $\{ {\mathcal P}_{12}^1, {\mathcal P}_{23}^1\} $ bounds  a Hopf band $H_{123}^1$,  i.e.  there is an embedding   $\beta:\mathbb S^1\times [0,1]\to {\mathbb E}_{a,b,c,d}$ such that ${\mathcal P}_{12}^1=\beta({\mathbb  S}^1\times \{0\})$ and  ${\mathcal P}_{23}^1=\beta({\mathbb S}^1\times \{1\})$  
  and the  surface $H_{123}^1$ is  foliated by leaves of the principal foliation ${\mathcal F}_2$.

   This structure of Hopf bands follows since the ellipsoid ${\mathbb E}_{a,b,c,d}$ with four different axes  belongs to a quadruply
orthogonal family of quadrics and 
its   
 principal lines are closed curves, obtained intersecting   three 
 hypersurfaces  
 of this family, see 
 the proof of 
 Lemma \ref{lem:cartaprincipal}.
 
Near each partially umbilic curve  $  {\mathcal P}_{12}^1,$ $  {\mathcal P}_{12}^2$, $ {\mathcal P}_{23}^1$ and $  {\mathcal P}_{23}^2$ the principal foliation ${\mathcal F}_2$ 
has its transversal structure  
equivalent to
the principal configuration of  
 a Darbouxian umbilic point $D_1$, see Figs.    \ref{fig:pumb3caso2},   \ref{fig:pumb3caso3}  and \ref{fig:conexao} and Propositions \ref{prop:eabcc1} and \ref{prop:eabcc2}.  So the umbilic separatrix surface of ${\mathcal F}_2$ of  the  partially umbilic curves $ {\mathcal P}_{12}^1$ and $ {\mathcal P}_{23}^1  $ is a cylinder having 
 these curves    
in its  
boundary. 
As the pair $\{ {\mathcal P}_{12}^1, {\mathcal P}_{23}^1\} $ is linked, it follows that this invariant cylinder by the principal foliation ${\mathcal F}_2$ is a Hopf 
band.

  In similar way the pair     $\{ {\mathcal P}_{12}^2, {\mathcal P}_{23}^2\}$
   bounds
    a Hopf band $H_{123}^2$
  which is foliated by leaves of ${\mathcal F}_2$, see Fig. \ref{fig:conexao}.
 \end{proof}

\section{Concluding Comments} \label{sec:CR}

In this work has been established the geometric structure  (principal configuration)  defined on an ellipsoid
in
$\mathbb R^4$
by its umbilic and partially  umbilic points and by  the integral foliations of  its principal plane and line fields.
The results have been   proved  analytically  and    illustrated  graphically.

The configurations established in Proposition \ref{prop:eaabb} and  Theorem \ref{th:eabcc3}, respectively,
for the bi and triaxial ellipsoids   $E_{a,a,b,b}$  and $E_{a,c,c,b}$, have no parallel in the literature.   The second case  establishing
the existence of
four umbilic points   which separate  four arcs of partially umbilic points,   altogether   forming   a  closed regular curve, as illustrated in figure \ref{fig:pumb3b},
maybe
  regarded as the most novel  contribution of  this paper.

The explanation for the quadriaxial ellipsoid  $\mathbb E_{a,b,c,d}$  in  subsection \ref{ss:eabcd} 
 makes explicit and  gives more analytic and geometric details not found  in
previous studies on this matter found in literature:  \cite{garcia-tese},
 \cite{garcia3}. 
More details  not formulated explicitly
 neither illustrated in previous 
 works   are provided in figures \ref{fig:pabcd} and \ref{fig:conexao}.

 A crucial
 step
 in the study of lines of principal curvature  on surfaces in $\mathbb R^3$
 was given
 by  Darboux  who in 1887
  established
the local   principal configurations at generic umbilic points on analytic surfaces.
See \cite{dar}
and its extension from analytic to  smooth generic surfaces   by Gutierrez and Sotomayor in  1982  \cite{gutso}.

The determination of the  local principal configurations around the generic singularities of the
principal configuration of a  smooth hypersurface in $\mathbb R^4$ was  achieved  by Garcia
in 1989
 \cite{garcia-tese}. See also \cite{garcia}  and \cite{garcia3}.  This  can be considered as the analogue   for $\mathbb R^4$  of Darboux result
 mentioned above for  $\mathbb R^3$.
  A new proof  of  Garcia's  result, which also allows an extension
  to study  the  generic bifurcations  of  the singularities  of principal configurations  in  families of hypersurfaces
  depending generically on one parameter will appear in   \cite{de-so-ga}.
  See  also \cite{debora_tese}.

The local and global  study of principal configurations of
surfaces in $\mathbb R^3$   
 has been 
developed in several directions. A
very partial  list of references is given as a sample:
 \cite{ga-me-so},
\cite{gas}, \cite{gho},  \cite{gutso},    \cite{lang}, \cite{melo},
\cite{ruas},   \cite{xav}, \cite{soto}.

Meanwhile,  for the case of 
 hypersurfaces in   $\mathbb R^4$
 it seems that,  besides    the few papers cited in this work,
a wide horizon of possibilities wait to be explored.

\vskip .7cm
\noindent {\bf Acknowledgments.}
 The first author was partially supported by a doctoral fellowship  CAPES/CNPq.
The second author participated in the    FAPESP Thematic Project
2008/02841-4 and has a  fellowship  CAPES PVNS at UNIFEI.  The
second and third authors are fellows of CNPq and participated of
the project CNPq  Proc. 476672/2009-00. The authors were supported
by Pronex/FAPEG/CNPq Proc. 2012 10 26 7000 803. 


\bibliographystyle{plain}

{ \newpage

\author{\noindent D\'ebora Lopes  \\ Departamento de Matem\'atica\\ Universidade Federal 
de Sergipe\\
Av. Marechal Rondon, s/n Jardim Rosa Elze - CEP 49100-000\\
São Crist\'ov\~ao, SE, Brazil}
 \email{deb@deboralopes.mat.br}

\vskip 0.7cm
\author{\noindent Jorge Sotomayor\\ Instituto de Matem\'atica e Estat\'{\i}stica \\
Universidade  de S\~ao Paulo\\
 Rua do Mat\~ao  1010,
Cidade Universit\'aria, CEP 05508-090,\\
S\~ao Paulo, S. P, Brazil}
 \email{sotp@ime.usp.br}

 \vskip 0.7cm

 \author{\noindent Ronaldo Garcia\\Instituto de Matem\'atica e Estat\'{\i}stica \\
Universidade Federal de Goi\'as\\ CEP 74001--970, Caixa Postal 131 \\
Goi\^ania, Goi\'as, Brazil}
 \email{ragarcia@ufg.br}
}
\end{document}